\documentclass[sn-mathphys]{sn-jnl}

\jyear{2022}%

\theoremstyle{thmstyleone}%
\usepackage{array}
\newtheorem{theorem}{Theorem}

\theoremstyle{thmstyletwo}%
\newtheorem{lemma}{Lemma}
\theoremstyle{thmstylethree}%
\newtheorem{definition}{Definition}%

\begin{document}
	
	\title[Sketching Algorithms for Low-Rank Tucker Approximation]{Practical Sketching Algorithms for Low-Rank Tucker Approximation of Large Tensors}

	\author[1]{\fnm{Wandi } \sur{Dong}}\email{15560159213@163.com}
	
	\author*[1]{\fnm{Gaohang} \sur{Yu}}\email{maghyu@163.com}
	
	\author[2,1,3]{\fnm{Liqun} \sur{Qi}}\email{liqun.qi@polyu.edu.hk}
	
	\author[4]{\fnm{Xiaohao} \sur{Cai}}\email{x.cai@soton.ac.uk}

	\affil[1]{\orgdiv{Department of Mathematics}, \orgname{Hangzhou Dianzi University}, \city{Hangzhou}, \postcode{310018}, \country{China}}
	
	\affil[2]{\orgname{Huawei Theory Research Lab}, \city{Hong Kong}, 
		\country{China}}
	
	\affil[3]{\orgdiv{Department of Applied Mathematics},  \orgname{Hongkong Polytechnic University}, \city{Hong Kong}, 
		\country{China}}

	\affil[4]{\orgdiv{School of Electronics and Computer Science}, \orgname{University of Southampton},  \city{Southampton}, \postcode{SO17 1BJ},  \country{UK}}
	
	
	\abstract{Low-rank approximation of tensors has been widely used in high-dimensional data analysis. It usually involves singular value decomposition (SVD) of large-scale matrices with high computational complexity. Sketching is an effective data compression and dimensionality reduction technique applied to the low-rank approximation of large matrices. This paper presents two practical randomized algorithms for low-rank Tucker approximation of large tensors based on sketching and power scheme, with a rigorous error-bound analysis. Numerical experiments on synthetic and real-world tensor data demonstrate the competitive performance of the proposed algorithms.}

	
	
	
	\maketitle

	\section{{Introduction}} \label{Sec:intro}
	In practical applications, high-dimensional data, such as color images, hyperspectral images and videos, often exhibit a low-rank structure. Low-rank approximation of tensors has become a general tool for compressing and approximating high-dimensional data and has been widely used in scientific computing, machine learning, signal/image processing, data mining, and many other fields \cite{Comon2014}. The classical low-rank tensor factorization models include, e.g., Canonical Polyadic decomposition (CP) \cite{Hitchcock1928,Kiers2000}, Tucker decomposition \cite{Tucker1963,Tucker1966,De Lathauwer et al2000}, Hierarchical Tucker (HT) \cite{Hackbusch and Kuhn2009,Grasedyck2010}, and Tensor Train decomposition (TT) \cite{Oseledets2011}.
	This paper focuses on low-rank Tucker decomposition, also known as the low multilinear rank approximation of tensors. When the target rank of Tucker decomposition is much smaller than the original dimensions, it will have good compression performance.
	For a given $N$th-order tensor $\mathcal{X}\in\mathbb{R}^{I_1\times I_2\times\ldots\times I_N}$, the low-rank Tucker decomposition can be formulated as the following optimization problem, i.e.,
	\begin{equation}\label{minf}
		\min_{\mathcal{Y}}\|\mathcal{ X}-\mathcal{Y}\|_F^2,
	\end{equation}
	where $\mathcal{Y}\in\mathbb{R}^{I_1\times I_2\times\ldots\times I_N}$, with $\text{rank}(Y_{(n)})\leq r_n$ for $n=1,2,\ldots,N$, $Y_{(n)}$ is the mode-$n$ unfolding matrix of $\mathcal{Y}$, and $r_n$ is the rank of the mode-$n$ unfolding matrix of $\mathcal{X}$.

	For the Tucker approximation of higher-order tensors, the most frequently used non-iterative algorithms are the improved algorithms for the higher-order singular value decomposition (HOSVD) \cite{Tucker1966}, the truncated higher-order SVD (THOSVD) \cite{De Lathauwer2000} and the sequentially truncated higher-order SVD (STHOSVD) \cite{Vannieuwenhoven et al2012}. Although the results of THOSVD and STHOSVD are usually sub-optimal, they can use as reasonable initial solutions for iterative methods such as higher-order orthogonal iteration (HOOI) \cite{De Lathauwer2000}. However, both algorithms rely directly on SVD when computing the singular vectors of intermediate matrices, requiring large memory and high computational complexity when the size of tensors is large.

	Strikingly, randomized algorithms can reduce the communication among different levels of memories and are parallelizable. In recent years, many scholars have become increasingly interested in randomized algorithms for finding approximation Tucker decomposition of large-scale data tensors \cite{Zhou et al2014,Che and Wei2019,Minster et al2020,Che et al2020,Che et al2021,sun2020,Malik and Becker2018,Ahmadi-Asl et al2021}. For example, Zhou et al. \cite{Zhou et al2014} proposed a randomized version of the HOOI algorithm for Tucker decomposition. Che and Wei \cite{Che and Wei2019} proposed an adaptive randomized algorithm to solve the multilinear rank of tensors. Minster et al. \cite{Minster et al2020} designed randomized versions of the THOSVD and STHOSVD algorithms, i.e., R-STHOSVD. {Sun et al. \cite{sun2020} presented a single-pass randomized algorithm to compute the low-rank Tucker approximation of tensors based on a practical matrix sketching algorithm for streaming data, see also  \cite{Tropp et al2019} for more details.} Regarding more randomized algorithms proposed for Tucker decomposition, please refer to \cite{Che et al2021,Che et al2020,Malik and Becker2018,Ahmadi-Asl et al2021} for a detailed  review of randomized algorithms for solving Tucker decomposition of tensors in recent years involving, e.g., random projection, sampling, count-sketch, random least-squares, single-pass, and multi-pass algorithms.
	
{This paper presents two efficient randomized algorithms for finding the low-rank Tucker approximation of tensors, i.e., Sketch-STHOSVD and sub-Sketch-STHOSVD summarized in Algorithms \ref{alg:Sketch-STHOSVD} and \ref{alg:sub-Sketch-STHOSVD}, respectively. The main contributions of this paper are threefold.
Firstly, we propose a new one-pass sketching algorithm (i.e., Algorithm \ref{alg:Sketch-STHOSVD}) for low-rank Tucker approximation, which can significantly improve the computational efficiency of STHOSVD. Secondly, we present a new matrix sketching algorithm (i.e., Algorithm \ref{alg:sub-Sketch}) by combining the two-sided sketching algorithm proposed by Tropp et al. \cite{Tropp et al2019} with subspace power iteration. Algorithm \ref{alg:sub-Sketch} can accurately and efficiently compute the low-rank approximation of large-scale matrices.  Thirdly, the proposed Algorithm \ref{alg:sub-Sketch-STHOSVD} can deliver a more accurate Tucker approximation than simpler randomized algorithms by combining the subspace power iteration. More importantly, sub-Sketch-STHOSVD can converge quickly for any data tensors and  independently of singular value gaps.}
	
The rest of this paper is organized as follows. Section \ref{Sec:pre} briefly introduces some basic notations, definitions, and tensor-matrix operations used in this paper and recalls some classical algorithms, including THOSVD, STHOSVD, and R-STHOSVD, for low-rank Tucker approximation. Our proposed two-sided sketching algorithm for STHOSVD is given in Section \ref{Sec:sketch}. In Section \ref{Sec:sketch-alg}, we present an improved algorithm with subspace power iteration. The effectiveness of the proposed algorithms is validated thoroughly in Section \ref{Sec:exp} by numerical experiments on synthetic and real-world data tensors. We conclude in Section \ref{Sec:con}.
	
	\section{{Preliminary}} \label{Sec:pre}
	\subsection{{Notations and basic operations}}
	Some common symbols used in this paper are shown in the following Table \ref{Tab:symbol}.
	
	\begin{table} [h]
		\centering
		\caption{Common symbols used in this paper.} \label{Tab:symbol}
		\begin{tabular}{cc}
			\toprule
			\multicolumn{1}{m{1.5cm}}{\centering Symbols} &
			\multicolumn{1}{m{6cm}}{\centering Notations}\\
			\midrule
			$a$&scalar\\
			${A}$	& matrix \\
			$\mathcal{X}$	& tensor  \\
			${X}_{(n)}$ & mode-$n$ unfolding matrix of $\mathcal{X}$ \\
			$\times_n$ & mode-$n$ product of tensor and matrix \\
			${I}_n$ & identity matrix with size $n \times n$ \\
			$\sigma_{i}({A})$ & the $i$th largest singular value of ${A}$\\
			${A}^\top$ & transpose of ${A}$\\
			${A}^\dagger$ & pseudo-inverse of ${A}$\\
			\bottomrule
		\end{tabular}
	\end{table}
	
	We denote an $N$th-order tensor $\mathcal{X}\in\mathbb{R}^{I_1\times I_2\times\ldots\times I_N}$ with entries given by $x_{i_1,i_2,...,i_N},1\leq i_n\leq I_n, n=1,2,...,N.$ The Frobenius norm of $\mathcal{X}$ is defined as $$\|\mathcal{X}\|_F=\sqrt{\sum_{i_1,i_2,...,i_N}^{I_1,I_2,...,I_N}x_{i_1,i_2,...,i_N}^2} \ .$$
	The mode-$n$ tensor-matrix multiplication is a frequently encountered operation in tensor computation. The mode-$n$ product of a tensor $\mathcal{ X}\in\mathbb{R}^{I_1\times I_2\times\ldots\times I_N}$ by a matrix $A\in\mathbb{R}^{K\times I_n}$ (with entries $a_{k,i_n}$) is denoted as $\mathcal{Y}=\mathcal{ X}\times_{n}A\in \mathbb{R}^{I_1\times...\times I_{n-1}\times K \times I_{n+1}\times ...\times I_N}$, with entries
	$$y_{i_1,...,i_{n-1},k,i_{n+1},...,i_N}=\sum_{i_n=1}^{I_n}x_{i_1,...,i_{n-1},i_n,i_{n+1},...,i_N}a_{k,i_n}.$$
	The mode-$n$ matricization of higher-order tensors is the reordering of tensor elements into a matrix. The columns of mode-$n$ unfolding matrix $X_{(n)}\in \mathbb{R}^{I_n\times(\prod_{N\neq n}I_N)}$ are the mode-$n$ fibers of $\mathcal{X}$. More specifically, a element $(i_1,i_2,...,i_N)$ of $\mathcal{X}$ is maps on a element $(i_n,j)$ of $X_{(n)}$, where
	\begin{equation}
		j=1+\sum_{k=1,k\neq n}^{N}[(i_k-1)\prod_{m=1,m\neq n}^{k-1}I_m].
		\nonumber
	\end{equation}
	Let the rank of mode-$n$ unfolding matrix $X_{(n)}$ is $r_n$, the integer array $(r_1,r_2,...,r_N)$ is Tucker-rank of $N$th-order tensor $\mathcal{X}$, also known as the multilinear rank. The Tucker decomposition of $\mathcal{X}$ with rank $(r_1,r_2,...,r_N)$ is expressed as
	\begin{equation}\label{eqn:2}
		\mathcal{X}=\mathcal{G}\times_1{U}^{(1)}\times_2{U}^{(2)}\ldots\times_N{U}^{(N)},
	\end{equation}
	where $\mathcal{G}\in\mathbb{R}^{r_1\times r_2\times\ldots\times r_N}$ is the core tensor, and $\{{U}^{(n)}\}_{n=1}^{N}$ with ${U}^{(n)}\in \mathbb{R}^{I_n\times r_n}$ is the mode-$n$ factor matrices. The graphical illustration of Tucker decomposition for a third-order tensor shows in Figure \ref{fig:tucker}. We denote an optimal rank-$(r_1,r_2,...,r_N)$ approximation of a tensor $\mathcal{X}$ as $\hat{\mathcal{X}}_{\rm opt}$, which is the optimal Tucker approximation by solving the minimization problem in \eqref{minf}.
	\begin{figure}[h]\label{tucker}
		\centering
		\includegraphics[width=0.5\linewidth]{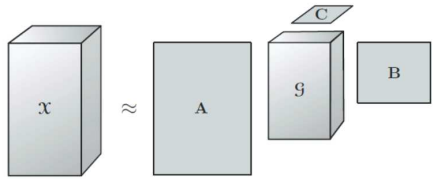}
		\caption{Tucker decomposition of a third-order tensor. }
		\label{fig:tucker}
	\end{figure}
	Below we present the definitions of some concepts used in this paper.

	\begin{definition}
		(Kronecker products) The Kronecker product of matrices $A\in \mathbb{R}^{m\times n}$ and $B\in \mathbb{R}^{k\times l}$ is defined as
		$$A\otimes B=\begin{bmatrix}
			a_{11}B & a_{12}B & ... & a_{1n}B \\
			a_{21}B & a_{22}B & ... & a_{2n}B \\
			\colon  & \colon  &\ddots &\colon  \\
			a_{m1}B & a_{m2}B & ... & a_{mn}B
		\end{bmatrix}\in \mathbb{R}^{mk\times nl}.$$
	\end{definition}
	The Kronecker product helps express Tucker decomposition. The Tucker decomposition in \eqref{eqn:2} implies
	$$X_{(n)}=U^{(n)}G_{(n)}(U^{(N)}\otimes...\otimes U^{(n+1)}\otimes U^{(n-1)}\otimes...\otimes U^{(1)} )^\top.$$

	\begin{definition}
		[Standard normal matrix] The elements of a standard normal matrix follow the real standard normal distribution (i.e., Gaussian with mean zero and variance one) form an independent family of standard normal random variables.
	\end{definition}
	
	\begin{definition}
		{(Standard Gaussian tensor) The elements of a standard Gaussian tensor follow the standard Gaussian distribution.}
	\end{definition}
	
	\begin{definition}
		[Tail energy] The $j$th tail energy of a matrix $X$ is defined as
		$$\tau_{j}^{2}({X}):=\min_{{\rm rank}({Y})<j}\|{X}-{Y}\|_{F}^{2}=\sum_{i\geq j}\sigma_{i}^{2}({X}).$$
	\end{definition}

	\subsection{Truncated higher-order SVD}
	Since the actual Tucker rank of large-scale higher-order tensor is hard to compute, the truncated Tucker decomposition with a pre-determined truncation $(r_1,r_2,...,r_N)$ is widely used in practice. THOSVD is a popular approach to computing the truncated Tucker approximation, also known as the best low multilinear rank-$(r_1,r_2,...,r_N)$ approximation, which reads
	
	\begin{equation}\nonumber
		\begin{aligned}
			\min_{\mathcal{G};{U}^{(1)},{U}^{(2)},\cdots,{U}^{(N)}}\|\mathcal{ X}-\mathcal{G}\times_1{U}^{(1)}\times_2{U}^{(2)}\cdots\times_N{U}^{(N)}\|_F^2\\
			{\rm s.t.}\quad {U}^{(n)\top}{U}^{(n)}={I}_{r_n}, n\in\{1,2,...,N\}.
		\end{aligned}
	\end{equation}
	
	\begin{algorithm}[H]
		\caption{THOSVD}
		\label{alg:THOSVD}
		\begin{algorithmic}[1]
			\Require  tensor $\mathcal{X}\in\mathbb{R}^{I_1\times I_2\times\ldots\times I_N}$ and target rank $(r_1,r_2,\ldots,r_N)$
			\Ensure Tucker approximation $\hat{\mathcal{X}}=\mathcal{G}\times_1{U}^{\left(1\right)}\times_1{U}^{\left(2\right)}\cdots\times_N{U}^{\left(N\right)}$
			
			\For   {$n = 1,2,\ldots,N$}
			\State    $ (U^{(n)},\sim,\sim)\gets {\tt truncatedSVD} ({X}_{(n)},r_n)$
			\EndFor
			\State   $\mathcal{G}\gets {\mathcal{X}\times}_1{U}^{\left(1\right)\top}\times_2{U}^{\left(2\right)\top}\cdots\times_N{U}^{\left(N\right)\top}$
		\end{algorithmic}
	\end{algorithm}

	Algorithm \ref{alg:THOSVD} summarizes the THOSVD approach. Each mode is processed individually in Algorithm \ref{alg:THOSVD}. The low-rank factor matrices of mode-$n$ unfolding matrix $X_{(n)}$ are computed through the truncated SVD,
	i.e.,
	$${X}_{(n)}=\begin{bmatrix}
		{U}^{(n)}	& \tilde{{U}^{(n)}}
	\end{bmatrix}\begin{bmatrix}
		S^{(n)}&  \\
		& \tilde{S^{(n)}}
	\end{bmatrix}\begin{bmatrix}
		{V}^{(n)\top}\\
		\tilde {{V}^{(n)\top}}
	\end{bmatrix}\cong{U}^{(n)} S^{(n)} {V}^{(n)\top},$$
	where ${U}^{(n)} S^{(n)} {V}^{(n)\top}$ is a rank-$r_n$ approximation of $X_{(n)}$, the orthogonal matrix ${U}^{(n)}\in \mathbb{R}^{I_n\times r_n}$ is the mode-${n}$ factor matrix of $\mathcal{ X}$ in Tucker decomposition, $ S^{(n)}\in\mathbb{R}^{r_n\times r_n}$ and ${V}^{(n)}\in \mathbb{R}^{ I_1...I_{n-1}I_{n+1}...I_N\times r_n}$. Once all factor matrices have been computed, the core tensor $\mathcal{G}$ can be computed as
	$$\mathcal{G}={\mathcal{X}\times}_1{U}^{\left(1\right)\top}\times_2{U}^{\left(2\right)\top}\cdots\times_N{U}^{\left(N\right)\top}\in \mathbb{R}^{r_1\times r_2\times ...\times r_N}.$$
	Then, the Tucker approximation $\hat{\mathcal{X}}$ of $\mathcal{X}$ can be computed as
	\begin{equation}
		\begin{split}
			\hat{\mathcal{X}} = & \ \mathcal{G}\times_1{U}^{(1)}\times_2{U}^{(2)}\cdots\times_N{U}^{(N)}  \\
			= & \ \mathcal{X}\times_1(U^{(1)}U^{(1)\top})\times_2(U^{(2)}U^{(2)\top})\cdots\times_N(U^{(N)}U^{(N)\top}).
		\end{split}
		\nonumber
	\end{equation}
	With the notation $\Delta_{n}^{2}(\mathcal{X}) \triangleq\sum_{i=r_{n}+1}^{I_{n}}\sigma_{i}^{2}({X}_{(n)})$ and $\Delta_{n}^2{(\mathcal{ X})\leq\|\mathcal{ X}-\hat{\mathcal{X}}_{\rm opt}\|_F^2}$ \cite{Minster et al2020}, the error-bound for Algorithm \ref{alg:THOSVD} can be stated in the following Theorem \ref{thmT}.

	\begin{theorem}[\cite{Vannieuwenhoven et al2012}, Theorem 5.1] \label{thmT}
		Let $\hat{\mathcal{X}}=\mathcal{G}\times_1U^{(1)}\times_2U^{(2)}\cdots\times_NU^{(N)}$ be the low multilinear rank-$(r_1,r_2,...,r_N)$ approximation of a tensor $\mathcal{X}\in\mathbb{R}^{I_1\times I_2\times\ldots\times I_N}$ by THOSVD. Then
		\begin{equation*}
			\begin{aligned}
				\|\mathcal{X}-\hat{\mathcal{X}}\|_{F}^{2}&\leq\sum_{n=1}^{N}\|\mathcal{X}\times_{n}({I}_{I_n}-{U}^{(n)}{U}^{(n)\top})\|_{F}^{2}=\sum_{n=1}^{N}\sum_{i=r_{n}+1}^{I_{n}}\sigma_{i}^{2}({X}_{(n)})\\
				&=\sum_{n=1}^{N}\Delta_{n}^{2}(\mathcal{X})\leq N\|\mathcal{ X}-\hat{\mathcal{X}}_{\rm opt}\|_F^2 .
			\end{aligned}
		\end{equation*}
	\end{theorem}

	\subsection{Sequentially truncated higher-order SVD}
	Vannieuwenhoven et al.\cite{Vannieuwenhoven et al2012} proposed one more efficient and less computationally complex approach for computing  approximate Tucker decomposition of tensors, called STHOSVD. Unlike THOSVD algorithm, STHOSVD updates the core tensor simultaneously whenever a factor matrix has computed.
	
	Given the target rank $(r_1,r_2,\ldots,r_N)$ and a processing order $s_p: \{1,2,...,N\}$, the minimization problem \eqref{minf} can be formulated as the following optimization problem
	\begin{equation}
		\begin{aligned}\label{pro3}
			&\min_{{U}^{(1)},\cdots,{U}^{(N)}}\|\mathcal{X}-\mathcal{X}\times_1(U^{(1)}U^{(1)\top})\times_2(U^{(2)}U^{(2)\top})\cdots\times_N(U^{(N)}U^{(N)\top})\|_F^2\\	
			= & \min_{{U}^{(1)},\cdots,{U}^{(N)}}(\|\mathcal{ X}\times_1(I_1-U^{(1)}U^{(1)\top})\|_F^2+\|\hat{\mathcal{ X}}^{(1)}\times_2(I_2-U^{(2)}U^{(2)\top})\|_F^2 + \\
			& \quad \quad \cdots +\|\hat{\mathcal{ X}}^{(N-1)}\times_N(I_N-U^{(N)}U^{(N)\top})\|_F^2)\\
			= & \min_{U^{(1)}}(\|\mathcal{ X}\times_1(I_1-U^{(1)}U^{(1)\top})\|_F^2 +\min_{U^{(2)}}(\|\hat{\mathcal{ X}}^{(1)}\times_2(I_2-U^{(2)}U^{(2)\top})\|_F^2 + \\
			& \quad \quad \min_{U^{(3)}}(\cdots +\min_{U^{(N)}}\|\hat{\mathcal{ X}}^{(N-1)}\times_N(I_N-U^{(N)}U^{(N)\top})\|_F^2))),
		\end{aligned}
	\end{equation}
	{where $\hat{\mathcal{X}}^{(n)}=\mathcal{X}\times_1(U^{(1)}U^{(1)\top})\times_2(U^{(2)}U^{(2)\top})\cdots\times_n(U^{(n)}U^{(n)\top}), n=1,2,...,N-1$, denote the intermediate approximation tensors. }
	
	\begin{algorithm}
		\caption{STHOSVD}\label{alg:STHOSVD}
		
		\begin{algorithmic}[1]
			\Require tensor $\mathcal{X}\in\mathbb{R}^{I_1\times I_2\times\ldots\times I_N}$, target rank $(r_1,r_2,\ldots,r_N)$, and processing order $s_p:\{i_1,i_2,\ldots,i_N\}$
			\Ensure Tucker approximation $\hat{\mathcal{X}}=\mathcal{G}\times_1{U}^{(1)}\times_2{U}^{(2)}\ldots\times_N{U}^{(N)}$
			\State	 $\mathcal{G}\gets\mathcal{X}$
			\For {$n = i_1,i_2,\ldots,i_N$}
			\State $ (U^{(n)},S^{(n)},V^{(n)\top})\gets {\tt truncatedSVD} ({G}_{(n)},r_n)$
			\State      $\mathcal{G}\gets {\tt fold_{n}}(S^{(n)}V^{(n)\top})$ (\% forming the updated tensor from its mode-$n$ unfolding)
			\EndFor
		\end{algorithmic}
	\end{algorithm}
	
	In Algorithm \ref{alg:STHOSVD}, the solution $U^{(n)}$ of problem (\ref{pro3}) can be obtained via $ {\tt truncatedSVD}(G_{(n)},r_n)$, where $G_{(n)}$ is mode-$n$ unfolding matrix of the $(n-1)$-th intermediate core tensor $\mathcal{G}=\mathcal{ X}\times _{i=1}^{n-1}U^{(i)\top
	}\in \mathbb{R}^{r_1\times r_2\times ...\times r_{n-1}\times I_{n}\times ...\times I_{N} }$, i.e.,
	$$G_{(n)}=\begin{bmatrix}
		{U}^{(n)}	& \tilde{{U}^{(n)}}
	\end{bmatrix}\begin{bmatrix}
		S^{(n)}&  \\
		& \tilde{S^{(n)}}
	\end{bmatrix}\begin{bmatrix}
		{V}^{(n)\top}\\
		\tilde {{V}^{(n)\top}}
	\end{bmatrix}\cong{U}^{(n)} S^{(n)} {V}^{(n)\top},$$
	where the orthogonal matrix ${U}^{(n)}$ is the mode-${n}$ factor matrix, and $S^{(n)}{V}^{(n)\top}\in \mathbb{R}^{r_n\times r_1...r_{n-1}I_{n+1}...I_N}$ is
	used to update the $n$-th intermediate core tensor $\mathcal{G}$. {Function ${\tt fold_{n}}(S^{(n)}V^{(n)\top})$ tensorizes matrix $S^{(n)}{V}^{(n)\top}$ into  tensor $\mathcal{G}\in \mathbb{R}^{r_1\times r_2\times ...\times r_{n}\times I_{n+1}\times ...\times I_{N} }$.} When the target rank $r_n$ is much smaller than $I_n$, the size of the updated intermediate core tensor $\mathcal{G}$ is much smaller than original tensor. This method can significantly improve computational performance. STHOSVD algorithm possesses the following error-bound.
	
	\begin{theorem}[\cite{Vannieuwenhoven et al2012}, Theorem 6.5]\label{thmST}
		Let $\hat{\mathcal{X}}=\mathcal{G}\times_1U^{(1)}\times_2U^{(2)}\ldots\times_NU^{(N)}$ be the low multilinear rank-$(r_1,r_2,...,r_N)$ approximation of a tensor $\mathcal{X}\in\mathbb{R}^{I_1\times I_2\times\ldots\times I_N}$ by STHOSVD with processsing order $s_p:\{1,2,\ldots,N\}$. Then
		\begin{equation*}
			\begin{aligned}
				\|\mathcal{X}-\hat{\mathcal{X}}\|_{F}^{2}&=\sum_{n=1}^{N}\|\hat{\mathcal{X}}^{(n-1)}-\hat{\mathcal{X}}^{(n)}\|_{F}^{2}\leq\sum_{n=1}^{N}\|\mathcal{X}\times_{n}({I}_{I_n}-{U}^{(n)}{U}^{(n)\top})\|_{F}^{2}\\
				&=\sum_{n=1}^{N}\Delta_{n}^{2}(\mathcal{X})\leq N\|\mathcal{X}-\hat{\mathcal{X}}_{\rm opt}\|_F^2.
			\end{aligned}
		\end{equation*}
	\end{theorem}
	
	Although STHOSVD has the same error-bound as THOSVD, it is less computationally complex and requires less storage. As shown in Section \ref{Sec:exp} for the numerical experiment, the running (CPU) time of the STHOSVD algorithm is significantly reduced, and the approximation error has slightly better than that of THOSVD in some cases.
	
	\subsection{Randomized STHOSVD}
	When the dimensions of data tensors are enormous, the computational cost of
	the classical deterministic algorithm TSVD for finding a low-rank approximation of mode-$n$ unfolding matrix can be expensive. Randomized low-rank matrix algorithms replace original large-scale matrix with a new one through a preprocessing step. The new matrix contains as much information as possible about the rows or columns of original data matrix. Its size is smaller than original matrix, allowing the data matrix to be processed efficiently and thus reducing the memory requirements for solving low-rank approximation of large matrix.
	
	\begin{algorithm}[H]
		\caption{R-SVD}
		\label{alg:R-SVD}
		\begin{algorithmic}[1]
			\Require  matrix ${A}\in\mathbb{R}^{m\times n}$, target rank $r$, and oversampling parameter $p\geq0$
			\Ensure low-rank approximation matrix $\hat{A}=\hat{U}\hat{S}\hat{V}^\top$ of $A$
			\State  $\Omega\gets {\tt randn}(n,r+p)$
			\State  ${Y}\gets {A}\Omega$
			\State $({Q},\sim)\gets {\tt thinQR}({Y})$
			\State ${B}\gets Q^{\top}A$
			\State {(${U},S, {V}^\top)\gets {\tt thinSVD}(B)$}
			\State $\hat{U}\gets QU(:,1:r)$
			\State $\hat{S}\gets S(1:r,1:r),\hat{V}\gets V(:,1:r)$
		\end{algorithmic}
	\end{algorithm}
	
	N. Halko et al. \cite{Halko et al2011} proposed a randomized SVD (R-SVD) for matrices. The preprocessing stage of the algorithm is performed by right multiplying original data matrix ${A}\in\mathbb{R}^{m\times n}$ with a random Gaussian matrix $\Omega\in \mathbb{R}^{n\times r}$. Each column of the resulting new matrix $Y=A\Omega\in \mathbb{R}^{m\times r}$ is a linear combination of the columns of original data matrix. When $r<n$, the size of matrix $Y$ is smaller than $A$. The oversampling technique can improve the accuracy of solutions. Subsequent computations are summarised in Algorithm \ref{alg:R-SVD}, where {\tt randn} generates a Gaussian random matrix, {\tt thinQR} produces an economy-size of the QR decomposition, and {\tt thinSVD} is the thin SVD decomposition. When $A$ is dense, the arithmetic cost of Algorithm \ref{alg:R-SVD} is $\mathcal{O}(2(r+p)mn+r^2(m+n))$ flops, where $p>0$ is the oversampling parameter satisfying $r+p\leq \min\{m,n\}$.
	
	Algorithm \ref{alg:R-SVD} is an efficient randomized algorithm for computing rank-$r$ approximations to matrices. Minster et al. \cite{Minster et al2020} applied Algorithm \ref{alg:R-SVD} directly to the STHOSVD algorithm and then presented a randomized version of STHOSVD (i.e., R-STHOSVD), see Algorithm \ref{alg:R-STHOSVD}.
	\begin{algorithm}[H]
		
		\caption{R-STHOSVD}
		\label{alg:R-STHOSVD}
		\begin{algorithmic}[1]
			\Require tensor $\mathcal{X}\in\mathbb{R}^{I_1\times I_2\times\ldots\times I_N}$, targer rank $(r_1,r_2,\ldots,r_N)$, processing order $s_p:\{i_1,i_2,\ldots,i_N\}$, and oversampling parameter $p\geq 0$
			\Ensure Tucker approximation $\hat{\mathcal{X}}=\mathcal{G}\times_1{U}^{\left(1\right)}\times_2{U}^{\left(2\right)}\ldots\times_N{U}^{\left(N\right)}$
			\State $\mathcal{G}\gets\mathcal{X}$
			\For{$n = i_1,i_2,\ldots,i_N$}
			\State  $(\hat{U}, \hat{S}, \hat{V}^\top)\gets$ \textbf{R-SVD}$({G}_{(n)}, r_n, p)$ (cf. Algorithm \ref{alg:R-SVD})
			\State  	  ${U}^{\left(n\right)}\gets {\hat{U}}$
			\State  $\mathcal{G}\gets {\tt fold_{n}}(\hat{S}\hat{V}^\top)$
			\EndFor
			
		\end{algorithmic}
	\end{algorithm}

	\section{Sketching algorithm for STHOSVD} \label{Sec:sketch}
	A drawback of R-SVD algorithm is that when both dimensions of the intermediate matrices are enormous, the computational cost can still be high. To resolve this problem, we could resort to the two-sided sketching algorithm for low-rank matrix approximation proposed by Joel A. Tropp et al. \cite{Tropp et al2017}. The preprocessing of sketching algorithm needs two sketch matrices to contain information regarding the rows and columns of input matrix ${A}\in\mathbb{R}^{m\times n}$. Thus we should choose two sketch size parameters $k$ and $l$, s.t. , $r\le k\le \min\{l, n\}$, $0< l\le m$. The random matrices $\Omega\in\mathbb{R}^{n\times k}$ and $\Psi \in\mathbb{R}^{l\times m}$ are fixed independent standard normal matrices.
	Then we can multiply matrix ${A}$ left and right respectively to obtain random sketch matrices ${Y}\in\ \mathbb{R}^{m\times k}$ and ${W}\in \mathbb{R}^{l\times n}$, which collect sufficient data about the input matrix to compute the low-rank approximation. The dimensionality and distribution of the random sketch matrices determine the approximation's potential accuracy, with larger values of $k$ and $l$ resulting in better approximations but also requiring more storage and computational cost.
	
	\begin{algorithm}
		\caption{{\bf Sketch} for low-rank approximation}
		\label{alg:Sketch}
		\begin{algorithmic}[1]
			\Require  matrix ${A}\in\mathbb{R}^{m\times n}$, and sketch size parameters $k,l$
			\Ensure  rank-$k$ approximation $\hat{{A}}={Q}{X}$ of $A$
			\State $\Omega\gets {\tt randn}(n,k), \Psi\gets {\tt randn}(l,m)$
			\State $\Omega\gets {\tt orth}(\Omega), {\Psi^\top \gets {\tt orth}(\Psi^\top)}$
			\State${Y}\gets {A} \Omega$
			\State${W}\gets\Psi {A}$
			\State $({Q},\sim)\gets {\tt thinQR}({Y})$
			\State${X}\gets (\Psi {Q})^\dag {W}$
		\end{algorithmic}
	\end{algorithm}
	
	The sketching algorithm for low-rank approximation is given in Algorithm \ref{alg:Sketch}. Function ${\tt orth}(A)$ in Step 2 produces an orthonormal basis of A. Using orthogonalization matrices will achieve smaller errors and better numerical stability than directly using the randomly generated Gaussian matrices. In particular, when $A$ is dense, the arithmetic cost of Algorithm \ref{alg:Sketch} is $\mathcal{O}((k+l)mn+kl(m+n))$ flops. Algorithm \ref{alg:Sketch} is simple, practical, and possesses the sub-optimal error-bound as stated in the following Theorem \ref{thmSketch}.
	In Theorem \ref{thmSketch}, function $f(s,t):=s/(t-s-1)(t>s+1>1)$. The minimum in Theorem \ref{thmSketch} reveals that the low rank approximation of given matrix $A$ automatically exploits the decay of tail energy.
	
	\begin{theorem} [\cite{Tropp et al2017}, Theorem 4.3] \label{thmSketch}
		Assume that the sketch size parameters satisfy $l>k+1$, and draw random test matrices $\Omega\in\mathbb{R}^{n\times k}$ and $\Psi{\in\mathbb{R}}^{l\times m}$ independently forming the standard normal distribution. Then the rank-$k$ approximation $\hat{{A}}$ obtained from Algorithm \ref{alg:Sketch} satisfies
		\begin{equation}\nonumber
			\begin{aligned}
				\mathbb{E}\parallel {A}-\hat{{A}}\parallel_F^2&\le(1+f(k,l))\cdot\min_{\varrho<k-1}(1+f(\varrho,k))\cdot\tau_{\varrho+1}^2({A})\\
				&=\frac{k}{l-k-1}\cdot\min_{\varrho<k-1}\frac{k}{k-\varrho-1}\cdot\tau_{\varrho+1}^2({A}).
			\end{aligned}
		\end{equation}
	\end{theorem}
	
	Using the two-sided sketching algorithm to leverage STHOSVD algorithm, we propose a practical sketching algorithm for STHOSVD named Sketch-STHOSVD. We summarize the procedures of Sketch-STHOSVD algorithm in Algorithm \ref{alg:Sketch-STHOSVD}, with its error analysis  stated in Theorem \ref{thm:Sketch-STHOSVD}.
	
	\begin{algorithm}[H]
		\caption{Sketch-STHOSVD}
		\label{alg:Sketch-STHOSVD}
		\begin{algorithmic}[1]
			\Require tensor $\mathcal{X}\in\mathbb{R}^{I_1\times I_2\times\ldots\times I_N}$, targer rank $(r_1,r_2,\ldots,r_N)$, processing order $s_p:\{i_1,i_2,\ldots,i_N\}$, and sketch size parameters $\{l_1,l_2,...,l_N\}$
			\Ensure Tucker approximation $\hat{\mathcal{X}}=\mathcal{G}\times_1{U}^{\left(1\right)}\times_2{U}^{\left(2\right)}\ldots\times_N{U}^{\left(N\right)}$
			\State $\mathcal{G}\gets\mathcal{X}$
			\For{ $n = i_1,i_2,\ldots,i_N$ }
			\State  $({Q},{X})\gets$ \textbf{Sketch}$({G}_{(n)},r_n,l_n)$ (cf. Algorithm \ref{alg:Sketch})
			\State  	  $  {U}^{\left(n\right)}\gets {Q}$
			\State  $\mathcal{G}\gets {\tt fold_{n}}(X)$
			\EndFor
			
		\end{algorithmic}
	\end{algorithm}
	
	\begin{theorem} \label{thm:Sketch-STHOSVD}
		Let $\hat{\mathcal{X}}=\mathcal{G}\times_1{U}^{(1)}\times_2{U}^{(2)}\ldots\times_N{U}^{(N)}$ be the Tucker approximation of a tensor $\mathcal{X}\in\mathbb{R}^{I_1\times I_2\times\ldots\times I_N}$ by the Sketch-STHOSVD algorithm (i.e., Algorithm \ref{alg:Sketch-STHOSVD}) with target rank $r_n<I_n, n=1,2,...,N$, sketch size parameters $\{l_1,l_2,...,l_N\}$ and processing order $s_p:\{1,2,\ldots,N\}$. Then
		\begin{equation*}
			\begin{aligned}
				\mathbb{E}_{\{\Omega _j\}_{j = 1}^N}\|\mathcal {X} - \widehat {\mathcal {X}}\|_F^2&\le  {\sum\limits_{n = 1}^N \frac{r_n}{l_n-r_n-1}\min_{\varrho_n<r_n-1}\frac{r_n}{r_n-\varrho_n-1}\Delta _n^2({\cal X})}\\
				&\le  \sum\limits_{n = 1}^N \frac{r_n}{l_n-r_n-1}\min_{\varrho_n<r_n-1}\frac{r_n}{r_n-\varrho_n-1}\|\mathcal{X}-\hat{\mathcal{X}}_{\rm opt}\|_F^2.
			\end{aligned}	
		\end{equation*}
	\end{theorem}
	\begin{proof}  	
		Combining Theorem \ref{thmST} and Theorem \ref{thmSketch}, we have
		\begin{equation*}
			\begin{aligned}
				& \ \mathbb{E}_{\{\Omega _j\}_{j = 1}^N}\|\mathcal {X} - \widehat {\mathcal {X}}\|_F^2 \\
				= & \sum \limits_{n = 1}^N\mathbb{E}_{\{\Omega _j\}_{j = 1}^N} \|\hat{\mathcal {X}}^{(n - 1)} - \hat{\mathcal {X}}^{(n )}\|_F^2\\
				= &  \sum \limits_{n = 1}^N \mathbb{E}_{\{\Omega _j\}_{j = 1}^{n-1}}\left\{\mathbb{E}_{\Omega _n}\|\hat{\mathcal {X}}^{(n - 1)} - \hat{\mathcal {X}}^{(n)}\|_F^2 \right\}\\
				= & \sum \limits_{n = 1}^N \mathbb{E}_{\{\Omega _j\}_{j = 1}^{n-1}}\left\{\mathbb{E}_{\Omega _n}\|\mathcal {G}^{(n - 1)} \times _{i = 1}^{n - 1} {{U}^{(i)}}{ \times _n}({I} - {{U}^{(n)}}{U}^{(n)\top})
				\|_F^2 \right\}\\
				\le &  \sum \limits_{n = 1}^N \mathbb{E}_{\{\Omega _j\}_{j = 1}^{n-1}}\left\{\mathbb{E}_{\Omega _n}\|({I} - {{U}^{(n)}}{U}^{(n)\top}){G}_{n}^{n-1})
				\|_F^2 \right\}\\
				\le &  \sum \limits_{n = 1}^N \mathbb{E}_{\{\Omega _j\}_{j = 1}^{n-1}}\frac{r_n}{l_n-r_n-1}\min_{\varrho_n<r_n-1}\frac{r_n}{r_n-\varrho_n-1}\sum \limits_{i = r_n+1}^{I_n}\sigma_{i}^{2}({G}_{(n)}^{(n-1)})\\
				\le &  \sum\limits_{n = 1}^N \mathbb{E}_{\{\Omega _j\}_{j = 1}^{n-1}} \frac{r_n}{l_n-r_n-1}\min_{\varrho_n<r_n-1}\frac{r_n}{r_n-\varrho_n-1} \Delta _n^2({\cal X})\\
				= & \sum\limits_{n = 1}^N \frac{r_n}{l_n-r_n-1}\min_{\varrho_n<r_n-1}\frac{r_n}{r_n-\varrho_n-1}\Delta _n^2({\cal X})\\
				\le &  {\sum\limits_{n = 1}^N \frac{r_n}{l_n-r_n-1}\min_{\varrho_n<r_n-1}\frac{r_n}{r_n-\varrho_n-1}} \|\mathcal{X}-\hat{\mathcal{X}}_{\rm opt}\|_F^2.
			\end{aligned}
		\end{equation*}
	\end{proof}
	
	We assume the processing order for STHOSVD, R-STHOSVD, and Sketch-STHOSVD algorithms is $s_p:\{1,2,...,N\}$. Table \ref{tab2} summarises the arithmetic cost of different algorithms for the cases related to the general higher-order tensor $\mathcal{ X}\in \mathbb{R}^{I_1\times I_{2}\times...\times I_N}$ with target rank $(r_1,r_2,\ldots,r_N)$ and the special cubic tensor $\mathcal{ X}\in \mathbb{R}^{I\times I\times...\times I}$ with target rank $(r,r,...,r)$. Here the tensors are dense and the target ranks $r_j\ll I_j, j = 1, 2, \ldots, N$.
	
	\begin{table}[h]
		\begin{center}
			\begin{minipage}{\textwidth}
				\caption{Arithmetic cost for the algorithms THOSVD, STHOSVD, R-STHOSVD, and the proposed Sketch-STHOSVD.} \label{tab2}
				\tiny{
					\begin{tabular*}
						{\textwidth}{@{\extracolsep{\fill}}lcccccc@{\extracolsep{\fill}}}
						\toprule%
						Algorithm & $\mathcal{ X}\in \mathbb{R}^{I_1\times I_{2}\times...\times I_N}$ & $\mathcal{ X}\in \mathbb{R}^{I\times I\times...\times I}$ \\
						\midrule
						THOSVD  & $\mathcal{O}(\sum \limits_{j=1}^{N}I_{j}I_{1:N}+\sum_{j=1}^{N}r_{1:j}I_{j:N})$ & $\mathcal{O}(NI^{N+1}+\sum \limits_{j=1}^{N}r^jI^{N-j+1})$\\
						STHOSVD   & $\mathcal{O}(\sum \limits_{j=1}^{N}I_{j}r_{1:j-1}I_{j:N}+\sum \limits_{j=1}^{N}r_{1:j}I_{j+1:N})$  & $\mathcal{O}(\sum \limits_{j=1}^{N}r^{j-1}I^{N-j+2}+r^jI^{N-j})$ \\
						R-STHOSVD   & $\mathcal{O}(\sum \limits_{j=1}^{N}r_{1:j}I_{j:N}+\sum \limits_{j=1}^{N}r_{1:j}I_{j+1:N})$  & $\mathcal{O}(\sum \limits_{j=1}^{N}r^{j}I^{N-j+1}+r^jI^{N-j})$\\
						Sketch-STHOSVD  &$\mathcal{O}(\sum \limits_{j=1}^{N}r_jl_j(I_{j}+r_{1:j-1}I_{j+1:N})+\sum \limits_{j=1}^{N}r_{1:j}I_{j+1:N})$   & $\mathcal{O}(\sum \limits_{j=1}^{N}rl(I+r^{j-1}I^{N-j})+r^jI^{N-j})$\\
						\botrule
				\end{tabular*}}
			\end{minipage}
		\end{center}
	\end{table}

	\section{{Sketching algorithm with subspace power iteration}} \label{Sec:sketch-alg}
	When the size of original matrix is very large or the singular spectrum of original matrix decays slowly, Algorithm \ref{alg:Sketch} may produce a poor basis in many applications. {Inspired by \cite{Rokhlin et al2009}, we suggest using the power iteration technique to enhance the sketching algorithm by replacing ${A}$ with $({A}{A}^\top)^{q}{A}$, where $q$ is a positive integer.} According to the SVD decomposition of matrix ${A}$, i.e., ${A}={U}{S}{V}^\top$, we know that $({A}{A}^\top)^{q}{A}={U}{S}^{2q+1}{V}^\top$. It can see that ${A}$ and $({A}{A}^\top)^{q}{A}$ have the same left and right singular vectors, but the latter has a faster decay rate of singular values, making its tail energy much smaller.

	\begin{algorithm}[H]
		\caption{Sketching algorithm with subspace power iteration ({\bf sub-Sketch})}
		\label{alg:sub-Sketch}
		\begin{algorithmic}[1]
			\Require matrix ${A}\in\mathbb{R}^{m\times n}$, sketch size parameters $k,l$, and integer $q>0$
			\Ensure  rank-$k$ approximation $\hat{{A}}={Q}{X}$ of $A$
			\State $\Omega\gets {\tt randn}(n,k), \Psi\gets {\tt randn}(l,m)$
			\State$\Omega\gets {\tt orth}(\Omega), \Psi^\top \gets {\tt orth}(\Psi^\top)$
			\State${Y}= {A}\Omega$, ${W}=\Psi {A}$
			\State $Q_0 \gets {\tt thinQR}(Y)$
			\For{$j=1, \ldots, q$}
			\State $\hat{{Y}}_j={A}^\top{Q_{j-1}}$
			\State$(\hat{{Q}}_j,\sim)\gets  {\tt thinQR}(\hat{{Y}}_j)$
			\State${Y}_j={A}\hat{{Q}}_j$
			\State$({Q}_j,\sim)\gets  {\tt thinQR}({Y}_j)$
			
			\EndFor
			\State ${Q}={Q}_q$
			\State$X\gets (\Psi Q)^\dag W$
		\end{algorithmic}
	\end{algorithm}

	Although power iteration can improve the accuracy of Algorithm \ref{alg:Sketch} to some extent, it still suffers from a problem, i.e., during the execution with power iteration, the rounding errors will eliminate all information about the singular modes associated with the singular values. To address this issue, we propose an improved sketching algorithm by orthonormalizing the columns of the sample matrix between each application of $A$ and $A^\top$, see Algorithm \ref{alg:sub-Sketch}. When $A$ is dense, the arithmetic cost of Algorithm \ref{alg:sub-Sketch} is $\mathcal{O}((q+1)(k+l)mn+kl(m+n))$ flops. Numerical experiments show that a good approximation can achieve with a choice of 1 or 2 for subspace power iteration parameter \cite{Halko et al2011}.

	\begin{algorithm}[H]
		
		\caption{sub-Sketch-STHOSVD}
		\label{alg:sub-Sketch-STHOSVD}
		\begin{algorithmic}[1]
			\Require tensor $\mathcal{X}\in\mathbb{R}^{I_1\times I_2\times\ldots\times I_N}$, targer rank $(r_1,r_2,\ldots,r_N)$, processing order $s_p:\{i_1,i_2,\ldots,i_N\}$, sketch size parameters $\{l_1,l_2,...,l_N\}$, and integer $q>0$
			\Ensure Tucker approximation $\hat{\mathcal{X}}=\mathcal{G}\times_1{U}^{\left(1\right)}\times_2{U}^{\left(2\right)}\ldots\times_N{U}^{\left(N\right)}$
			\State $\mathcal{G}\gets\mathcal{X}$
			\For{ $n = i_1,i_2,\ldots,i_N$}
			\State  $ ({Q},{X})\gets$ \textbf{sub-Sketch}$({G}_{(n)},r_n,l_n,q)$ (cf. Algorithm \ref{alg:sub-Sketch})
			\State  	  $  {U}^{\left(n\right)}\gets {Q}$
			\State  $\mathcal{G}\gets {\tt fold_{n}}(X)$
			\EndFor
		\end{algorithmic}
	\end{algorithm}
	
	Using Algorithm \ref{alg:sub-Sketch} to compute the low-rank approximations of intermediate matrices, we can obtain an improved sketching algorithm for STHOSVD, called sub-Sketch-STHOSVD, see Algorithm \ref{alg:sub-Sketch-STHOSVD}.
	The error-bound for Algorithm \ref{alg:sub-Sketch-STHOSVD} states in the following Theorem \ref{thmsubs}. Its proof is deferred in Appendix.
	
	\begin{theorem}\label{thmsubs}
		Let $\hat{\mathcal{X}}=\mathcal{G}\times_1{U}^{(1)}\times_2{U}^{(2)}\ldots\times_N{U}^{(N)}$ be the Tucker
		approximation of a tensor $\mathcal{X}\in\mathbb{R}^{I_1\times I_2\times\ldots\times I_N}$ obtained by the sub-Sketch-STHOSVD algorithm (i.e., Algorithm \ref{alg:sub-Sketch-STHOSVD}) with target rank $r_n<I_n, n=1,2,...,N$, sketch size parameters $\{l_1,l_2,...,l_N\}$ and processing order $p:\{1,2,\ldots,N\}$. Let ${\varpi_k}\equiv \frac{\sigma_{k+1}}{\sigma_{k}}$ denote the singular value gap, then
		\begin{equation*}
			\begin{aligned}
				\mathbb{E}_{\{\Omega _j\}_{j = 1}^N}\|\mathcal {X} - \widehat {\mathcal {X}}\|_F^2&\le  {\sum\limits_{n = 1}^N (1+f(r_n,l_n))\cdot\min_{\varrho_n<r_n-1}(1+f(\varrho_n,r_n){\varpi_r}^{4q})\cdot\tau_{\varrho+1}^2({X_{(n)}})}\\
				&\le  \sum\limits_{n = 1}^N (1+f(r_n,l_n))\cdot\min_{\varrho_n<r_n-1}(1+f(\varrho_n,r_n){\varpi_r}^{4q})\|\mathcal{X}-\hat{\mathcal{X}}_{\rm opt}\|_F^2.
			\end{aligned}	
		\end{equation*}
	\end{theorem}
	\begin{proof}
		See Appendix.
	\end{proof}

	\section{{Numerical experiments}} \label{Sec:exp}
	
	This section conducts numerical experiments with synthetic data and real-world data, including comparisons between the traditional THOSVD, STHOSVD algorithms, the R-STHOSVD algorithm proposed in \cite{Minster et al2020}, and our proposed algorithms Sketch-STHOSVD and sub-Sketch-STHOSVD. Regarding the numerical settings, the oversampling parameter $p=5$ is used in Algorithm \ref{alg:R-SVD}, the sketch parameters $l_n = r_n+2, n = 1, 2, \ldots, N$, are used in Algorithms \ref{alg:Sketch-STHOSVD} and \ref{alg:sub-Sketch-STHOSVD}, and the power iteration parameter $q=1$ is used in Algorithm \ref{alg:sub-Sketch-STHOSVD}.

	\subsection{{Hilbert tensor}}
	Hilbert tensor is a synthetic and supersymmetric tensor, with each entry defined as
	$$\mathcal {X}_{{i_1}{i_2}...{i_n}}= \frac{1}{{{i_1} + {i_2} + ... + {i_n}}}, 1 \le {i_n} \le {I_n}, n = 1,2,...,N.$$
	In {the first} experiment, we set $N=5$ and $I_n = 25, n=1,2,\ldots,N$. The target rank is chosen as $(r,r,r,r,r)$, where $r\in[1,25]$. Due to the supersymmetry of the Hilbert tensor, the processing order in the algorithms does not affect the final experimental results, and thus the processing order can be directly chosen as $s_p:\{1,2,3,4,5\}$.
	\begin{figure}[htb]
		\centering
		\includegraphics[width=1\linewidth]{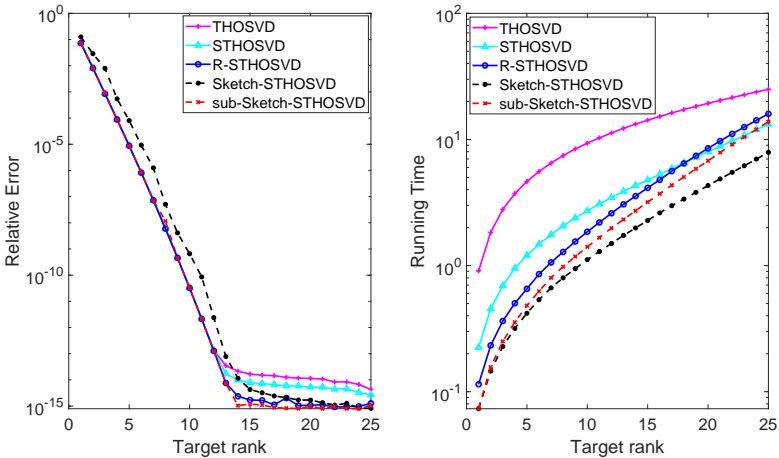}
		\caption{Results comparison on the Hilbert tensor with a size of $25\times25\times25\times25\times25$ in terms of numerical error (left) and CPU time (right).}
		\label{fig:results-hilbert-tensor}
	\end{figure}
	
	The results of different algorithms are given in Figure \ref{fig:results-hilbert-tensor}. It shows that our proposed algorithms (i.e., Sketch-STHOSVD and sub-Sketch-STHOSVD) and algorithm R-STHOSVD outperform the algorithms THOSVD and STHOSVD. In particular, the error of the proposed algorithms Sketch-STHOSVD and sub-Sketch-STHOSVD is comparable to R-STHOSVD (see the left plot in Figure \ref{fig:results-hilbert-tensor}), while they both use less CPU time than R-STHOSVD (see the right plot in Figure \ref{fig:results-hilbert-tensor}). This result demonstrates the excellent performance of the proposed algorithms and indicates that the two-sided sketching method and the subspace power iteration used in our algorithms can indeed improve the performance of STHOSVD algorithm.

	{For a large-scale test, we use a Hilbert tensor with a size of $500\times 500\times 500$ and conduct experiments using ten different approximate multilinear ranks. We perform the tests ten times and report the algorithms' average running time and relative error in Table \ref{hilbert-time} and Table \ref{hilbert-error}, respectively. The results show that the randomized algorithms can achieve higher accuracy than the deterministic algorithms. The proposed Sketch-STHOSVD algorithm is the fastest, and the sub-Sketch-STHOSVD algorithm achieves the highest accuracy efficiently.}

	\begin{table}[h]
		\begin{center}
			\begin{minipage}{\textwidth}
				\caption{{Results comparison in terms of the CPU time (in second) on the Hilbert tensor with a size of $500\times 500\times 500$ as the target rank increases.}} \label{hilbert-time}
				\tiny{
					\begin{tabular*}
						{\textwidth}{@{\extracolsep{\fill}}lcccccc@{\extracolsep{\fill}}}
						\toprule%
						Target rank & THOSVD  & STHOSVD & R-STHOSVD &Sketch-STHOSVD  &sub-Sketch-STHOSVD\\
						\midrule
						(10,10,10)& 17.18
						& 7.49
						& 0.92
						&\textbf{0.86}
						&0.98
						\\
						(20,20,20) & 23.13
						& 8.87
						& 1.25
						&\textbf{1.05}
						&1.48
						\\
						(30,30,30)& 24.91
						& 9.35
						& 1.66
						&\textbf{1.53}
						&2.16
						\\
						(40,40,40) & 28.05
						& 10.41
						& 1.94
						&\textbf{1.44}
						&2.11
						\\
						(50,50,50)& 29.44
						& 11.39
						& 2.07
						&\textbf{1.67}
						&2.43
						\\
						(60,60,60)& 30.14
						& 11.07
						& 2.37
						&\textbf{1.90}
						&2.77
						\\
						(70,70,70)& 29.44
						& 11.18
						& 2.57
						&\textbf{2.10}
						&3.02
						\\
						(80,80,80)& 29.65
						& 12.30
						& 3.05
						&\textbf{2.54}
						&3.75
						\\
						(90,90,90)& 31.11
						& 12.80
						& 3.80
						&\textbf{2.80}
						&4.33
						\\
						(100,100,100)& 32.22 & 13.51
						& 4.04
						&\textbf{3.07}
						&4.61
						\\
						\botrule
				\end{tabular*}}
			\end{minipage}
		\end{center}
	\end{table}
	
	\begin{table}[h]
		\begin{center}
			\begin{minipage}{\textwidth}
				\caption{{Results comparison in terms of the relative error on the Hilbert tensor with a size of $500\times 500\times 500$ as the target rank increases.}} \label{hilbert-error}
				\tiny{
					\begin{tabular*}
						{\textwidth}{@{\extracolsep{\fill}}lcccccc@{\extracolsep{\fill}}}
						\toprule%
						Target rank & THOSVD  & STHOSVD & R-STHOSVD &Sketch-STHOSVD  &sub-Sketch-STHOSVD\\
						\midrule
						(10,10,10)& 2.7354e-06
						& \textbf{2.7347e-06} & \textbf{2.7347e-06} &1.1178e-05 &2.7568e-06\\
						(20,20,20) & 1.1794e-12  & \textbf{1.1793e-12} & 1.1794e-12 &7.1408e-12  &1.2677e-12\\
						(30,30,30)& 4.6574e-15  & 3.2739e-15 & 3.2201e-15 &4.0641e-15  &\textbf{2.0182e-15} \\
						(40,40,40) & 4.4282e-15  & 3.4249e-15 & 2.8212e-15 &2.1562e-15   &\textbf{1.7860e-15} \\
						(50,50,50)& 4.1628e-15  & 3.2342e-15 & 2.6823e-15  &2.3205e-15  &\textbf{1.8625e-15} \\
						(60,60,60)& 4.1214e-15  & 3.1271e-15 & 2.3652e-15  &2.2920e-15  &\textbf{1.7472e-15} \\
						(70,70,70)& 4.1085e-15  & 3.0000e-15 & 2.1761e-15  &2.0499e-15   &\textbf{1.6370e-15} \\
						(80,80,80)& 4.0956e-15  & 3.1350e-15 & 1.8382e-15  &1.8209e-15   &\textbf{1.6424e-15} \\
						(90,90,90)& 4.0792e-15  & 3.3742e-15 & 1.8102e-15  & 1.7193e-15   &\textbf{1.5264e-15} \\
						(100,100,100)& 4.0390e-15  & 3.0571e-15 & 1.7323e-15  & 1.6304e-15   &\textbf{1.4957e-15} \\
						\botrule
				\end{tabular*}}
			\end{minipage}
		\end{center}
	\end{table}
	
	\subsection{{Sparse tensor}}
	
	\begin{figure}[htb]
		\centering
		\includegraphics[trim={{.22\linewidth} {.01\linewidth} {.25\linewidth} {.065\linewidth}}, clip,width=0.8\linewidth]{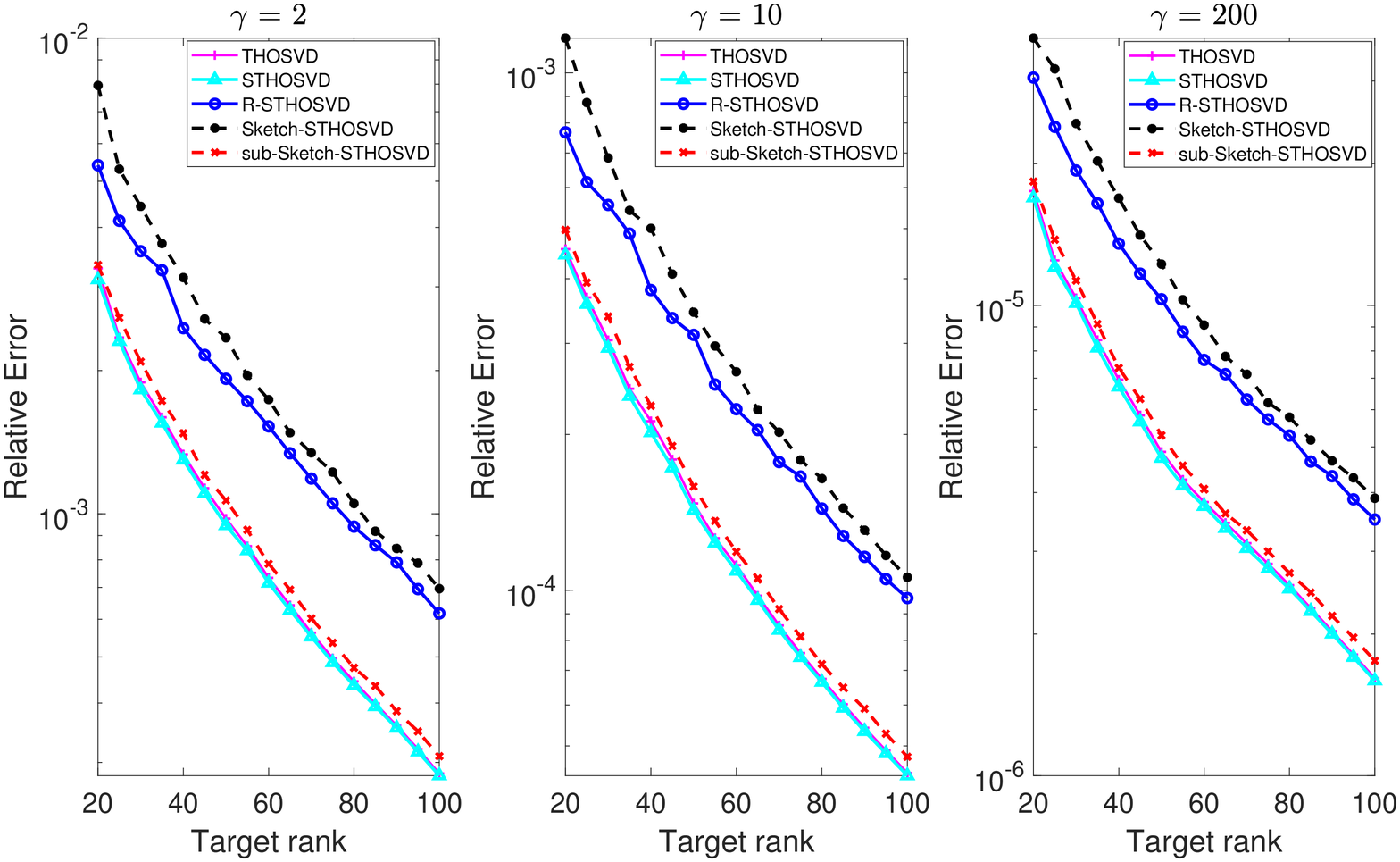} \\
		\includegraphics[trim={{.22\linewidth} {.01\linewidth} {.25\linewidth} {.065\linewidth}}, clip,width=0.8\linewidth]{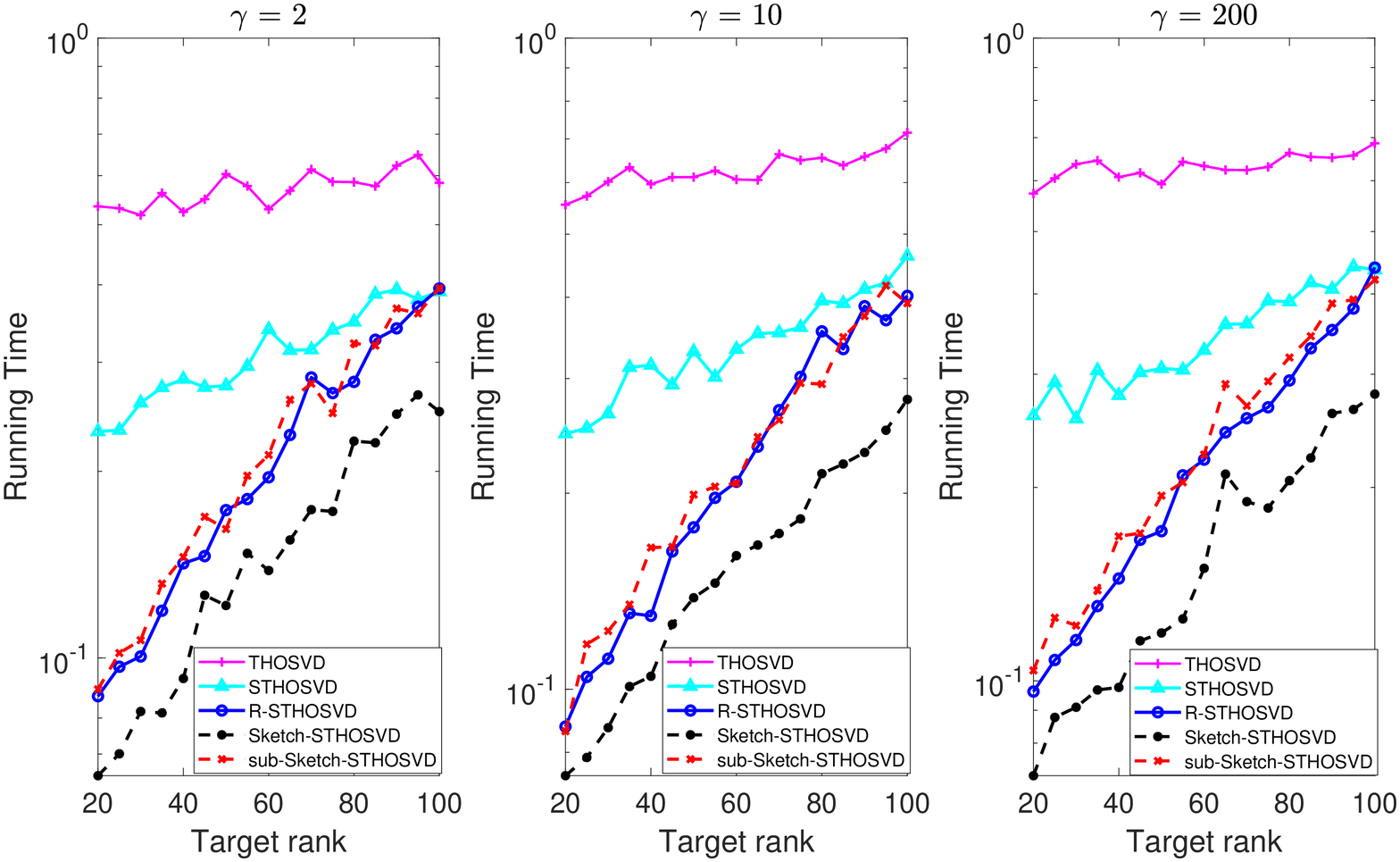}
		\caption{Results comparison on a sparse tensor with a size of $200\times200\times200$ in terms of numerical error (first row) and CPU time (second row). }
		\label{fig:results-sparse-tensor}
	\end{figure}
	In this experiment, we test the performance of different algorithms on a sparse tensor $\mathcal {X} \in \mathbb{R}^{200 \times 200 \times 200}$, i.e.,
	$$
	\mathcal {X} = \sum\limits_{i = 1}^{10} {\frac{\gamma }{{{i^2}}}} {{\bf{x}}_i} \circ {{\bf{y}}_i} \circ {{\bf{z}}_i} + \sum\limits_{i = 11}^{200} {\frac{1}{{{i^2}}}} {{\bf{x}}_i} \circ {{\bf{y}}_i} \circ {{\bf{z}}_i}.
	$$
	Where ${\bf{x}}_i,{{\bf{y}}_i},{{\bf{z}}_i} \in \mathbb{R}^n$ are sparse vectors all generated using the {\tt sprand} command in \textsc{MATLAB} with $5\%$ nonzeros each, and $\gamma$ is a user-defined parameter which determines the strength of the gap between the first ten terms and the rest terms. The target rank is chosen as $(r,r,r)$, where $r\in[20,100]$. The experimental results show in Figure \ref{fig:results-sparse-tensor}, in which three different values $\gamma = 2,10,200$ are tested. The increase of gap means that the tail energy will be reduced, and the accuracy of the algorithms will be improved. Our numerical experiments also verified this result. 
	
	Figure \ref{fig:results-sparse-tensor} demonstrates the superiority of the proposed sketching algorithms. In particular, we see that the proposed Sketch-STHOSVD is the fastest algorithm, with a comparable error against R-STHOSVD; the proposed sub-Sketch-STHOSVD can reach the same accuracy as the STHOSVD algorithm but in much less CPU time; and the proposed sub-Sketch-STHOSVD achieves much better low-rank approximation than R-STHOSVD with similar CPU time.
		\begin{figure}[htb]
		\centering
		\includegraphics[trim={{.25\linewidth} {.03\linewidth} {.25\linewidth} {.05\linewidth}}, clip,width=0.8\linewidth]{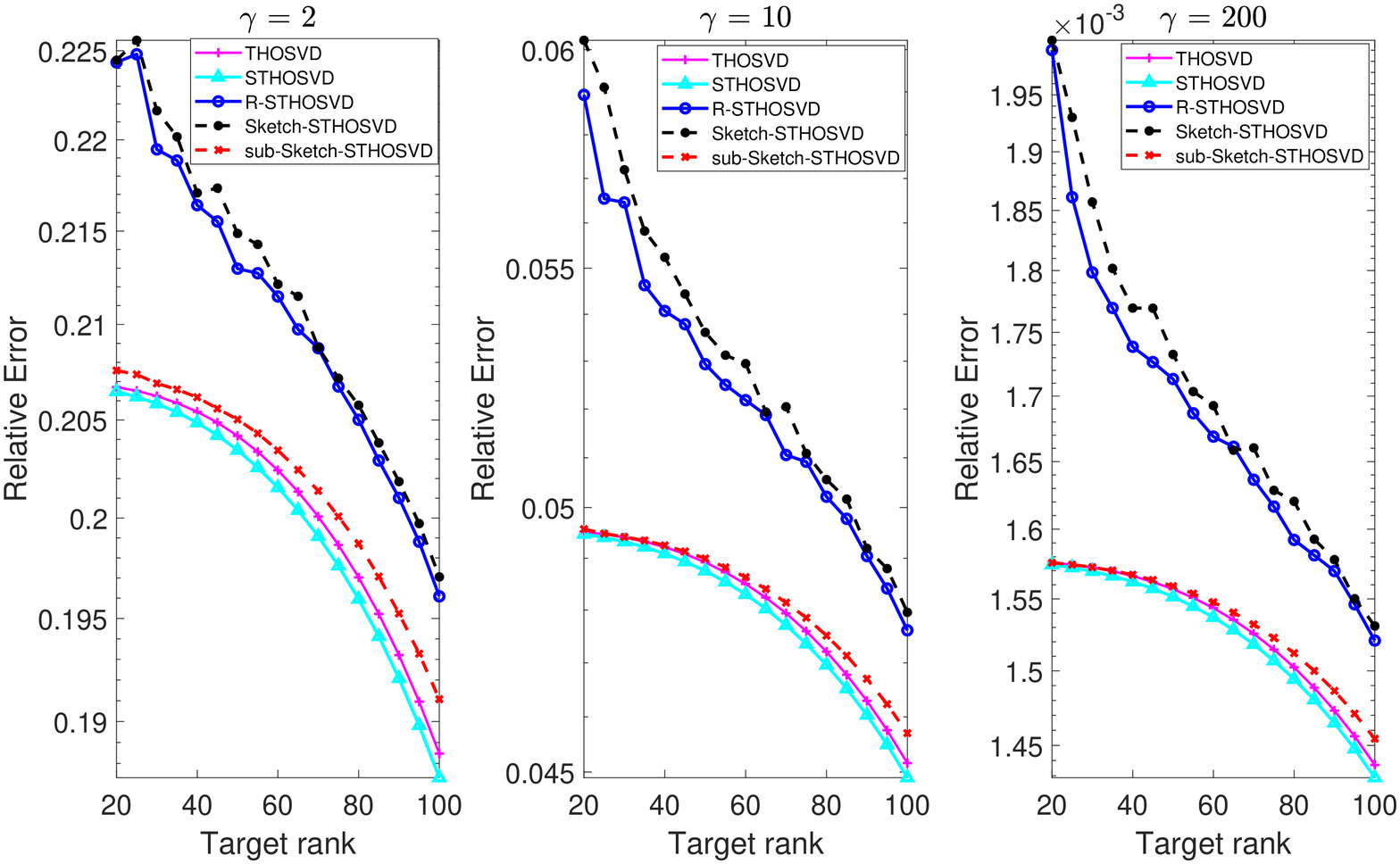} \\
		\includegraphics[trim={{.22\linewidth} {.03\linewidth} {.25\linewidth} {.065\linewidth}}, clip,width=0.8\linewidth]{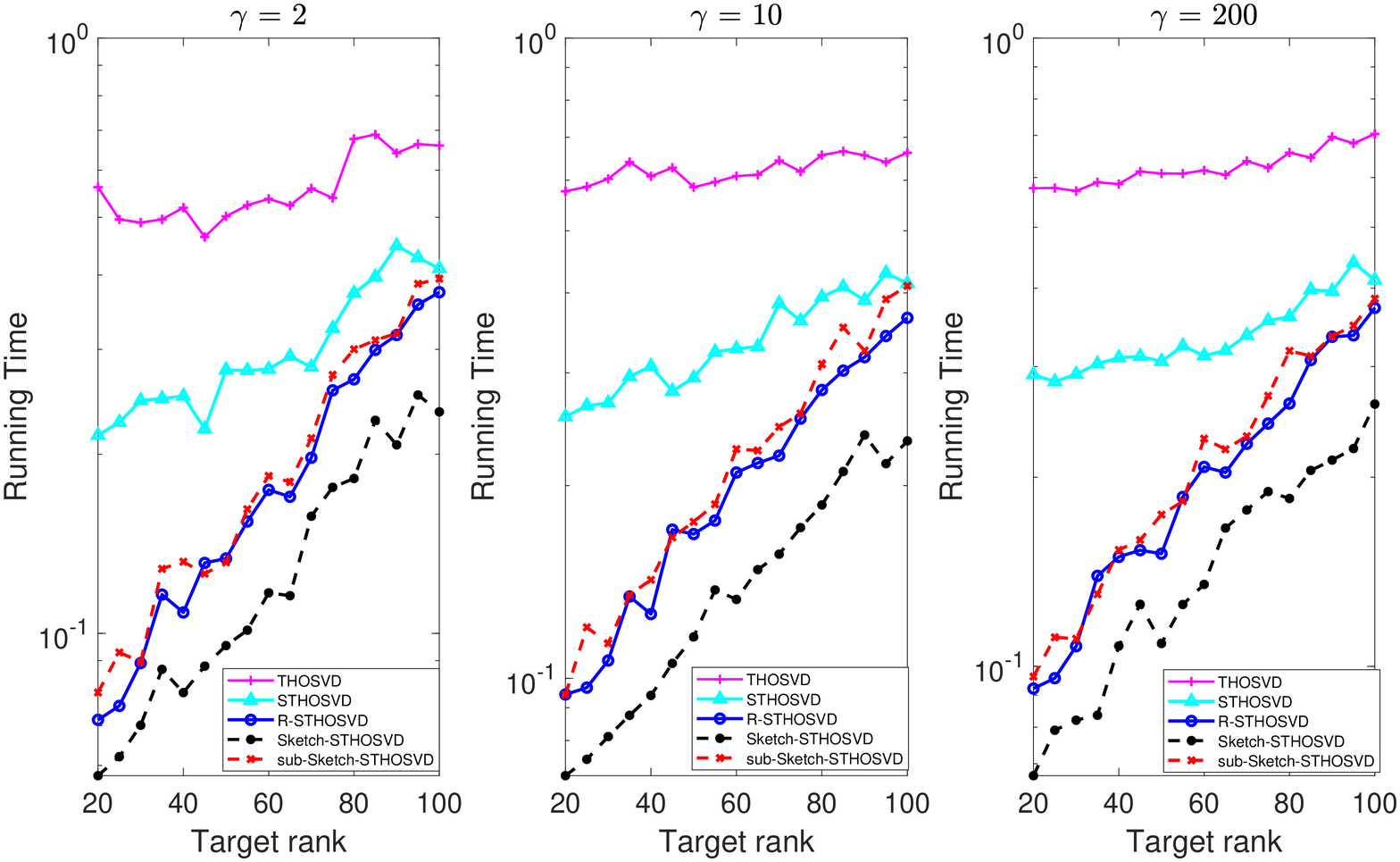}
		\caption{{Results comparison on a $200\times200\times200$ sparse tensor with noise in terms of numerical error (first row) and CPU time (second row). }}
		\label{fig:results-sparse-noise-tensor}
	\end{figure}
	
	{Now we consider the influence of noise on algorithms' performance. Specifically, the sparse tensor $\mathcal{X}$ with noise is designed in the same manner as in \cite{Xiao et al2021}, i.e.,}
	{$$\hat{\mathcal{X}}=\mathcal{X}+\delta\mathcal{K},$$
	where $\mathcal{K}$ is a standard Gaussian tensor and $\delta$ is used to control the noise level. Let $\delta=10^{-3}$ and keep the rest parameters the same as the settings in the previous experiment. The relative error and running time of different algorithms are shown in Figure \ref{fig:results-sparse-noise-tensor}.}
	{In Figure \ref{fig:results-sparse-noise-tensor}, we see that noise indeed affects the accuracy of the low-rank approximation, especially when the gap is small. However, the influence of noise does not change the conclusion obtained on the case without noise. The accuracy of our sub-Sketch-STHOSVD algorithm is the highest among the randomized algorithms. As $\gamma$ increases, sub-Sketch-STHOSVD can achieve almost the same accuracy as that of THOSVD and STHOSVD in a comparable CPU time against R-STHOSVD. }
	\subsection{{Real-world data tensor}}
	In this experiment, we test the performance of different algorithms on a colour image, called HDU picture\footnote{https://www.hdu.edu.cn/landscape}, with a size of $1200\times1800\times3$. We also evaluate the proposed sketching algorithms on the widely used YUV Video Sequences\footnote{http://trace.eas.asu.edu/yuv/index.html}. Taking the `hall monitor' video as an example and using the first 30 frames, a three order tensor with a size of $144\times176\times30$ is then formed for this test.

	Firstly, we conduct an experiment on the HDU picture with target rank $(500,500,3)$, and compare the PSNR and CPU time of different algorithms. The experimental result is shown in Figure \ref{fig:results-HDU-tensor}, which shows that the PSNR of sub-Sketch-STHOSVD, THOSVD and STHOSVD is very similar (i.e., $\sim$ 40) and that sub-Sketch-STHOSVD is more efficient in terms of CPU time. R-STHOSVD and Sketch-STHOSVD are also very efficient compared to sub-Sketch-STHOSVD; however, the PSNR they achieve is 5 dB less than sub-Sketch-STHOSVD.
	Then we conduct separate numerical experiments on the HDU picture and the `hall monitor' video clip as the target rank increases, and compare these algorithms in terms of the relative error, CPU time and PSNR, see Figure \ref{fig:results-HDU} and Figure \ref{fig:results-YUV-tensor}. These experimental results again demonstrate the superiority (i.e., low error and good approximation with high efficiency) of the proposed sub-Sketch-STHOSVD algorithm in computing the Tucker decomposition approximation.
		\begin{figure}[htb]
		\centering
		\begin{tabular}{ccc}
			\includegraphics[width=0.3\linewidth]{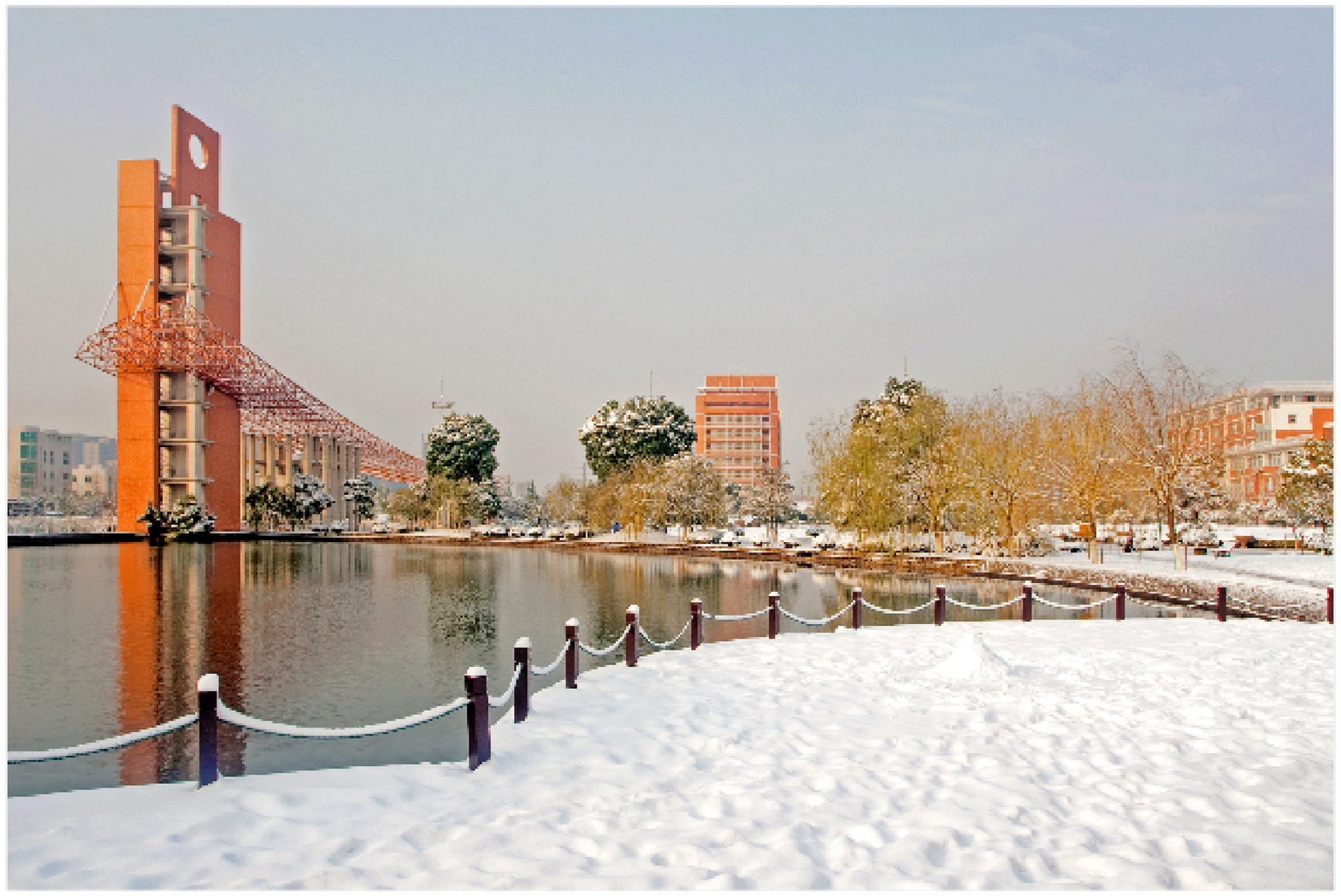} &
			\includegraphics[width=0.3\linewidth]{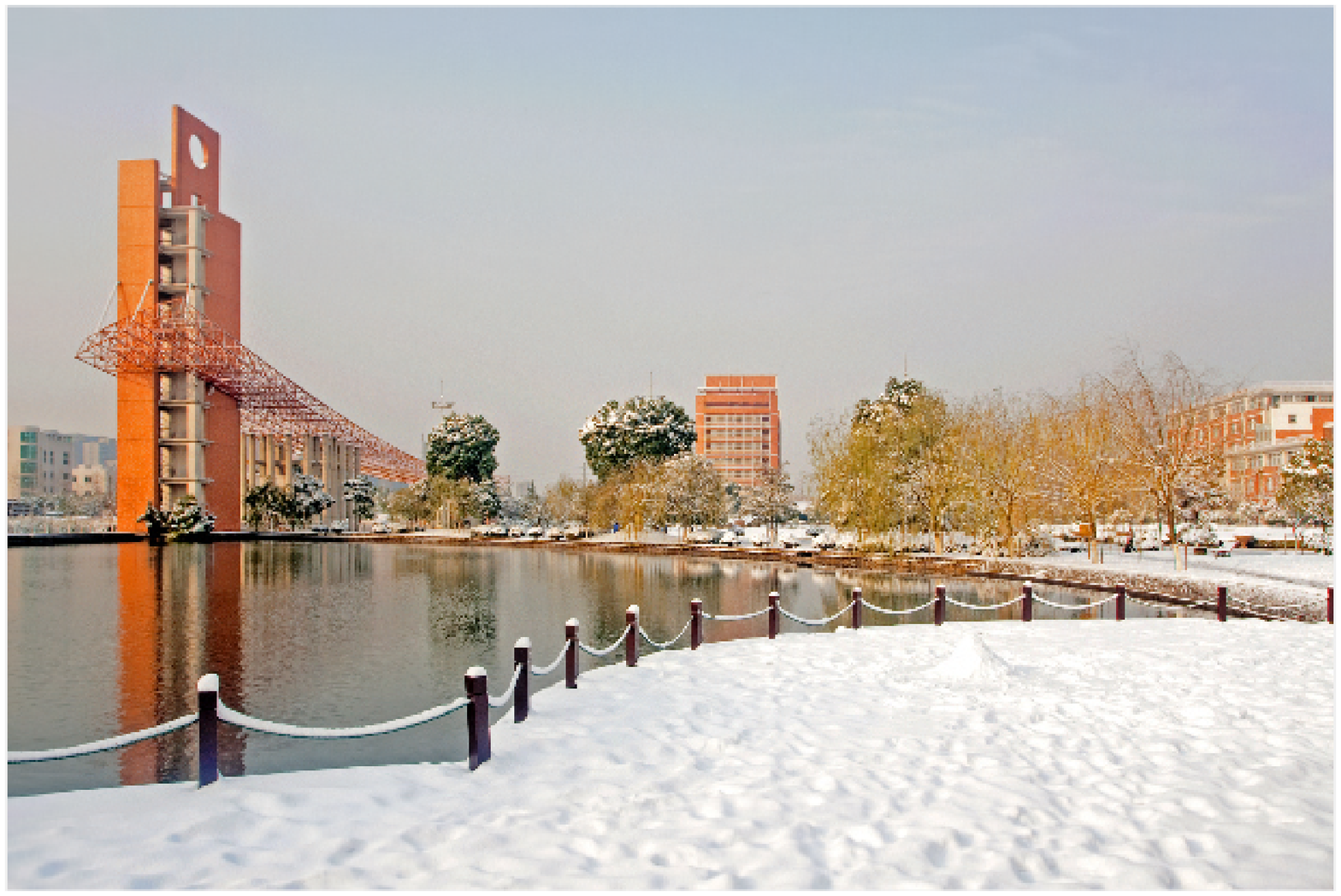} &
			\includegraphics[width=0.3\linewidth]{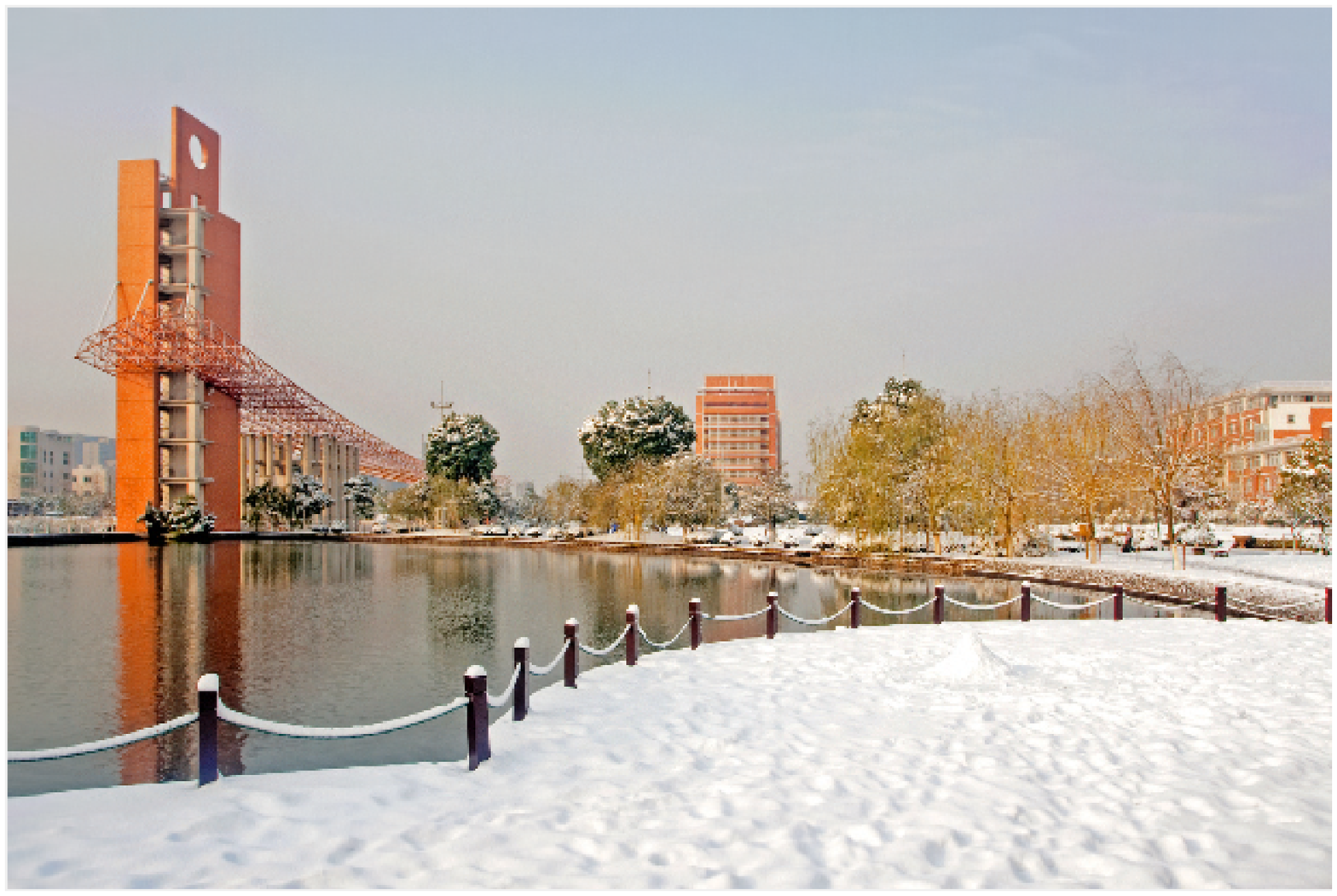} \\
			{\footnotesize Original} & {\footnotesize THOSVD (2.62; 40.61)} & {\footnotesize STHOSVD (1.89; 40.65)} \\
			\includegraphics[width=0.3\linewidth]{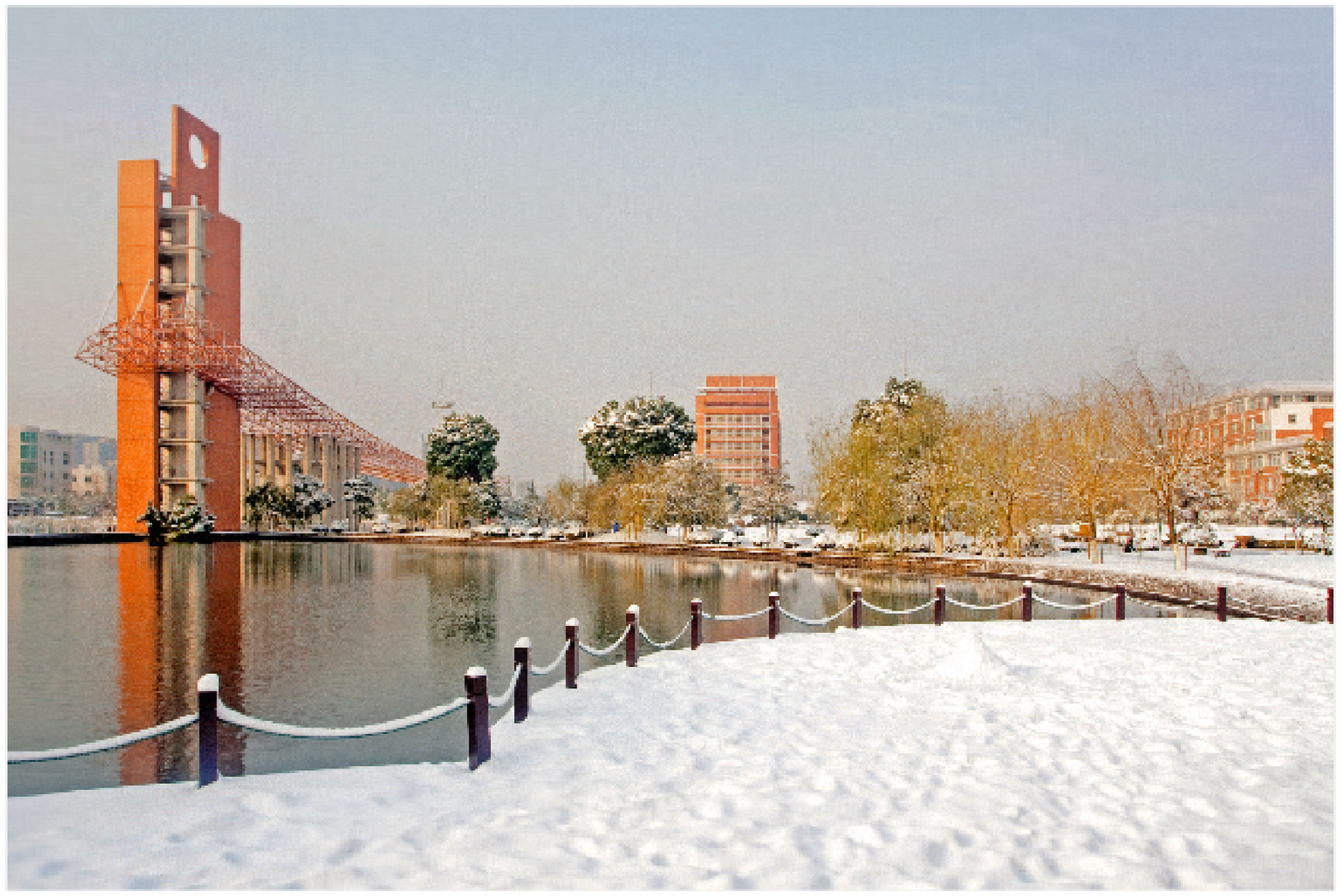} &
			\includegraphics[width=0.3\linewidth]{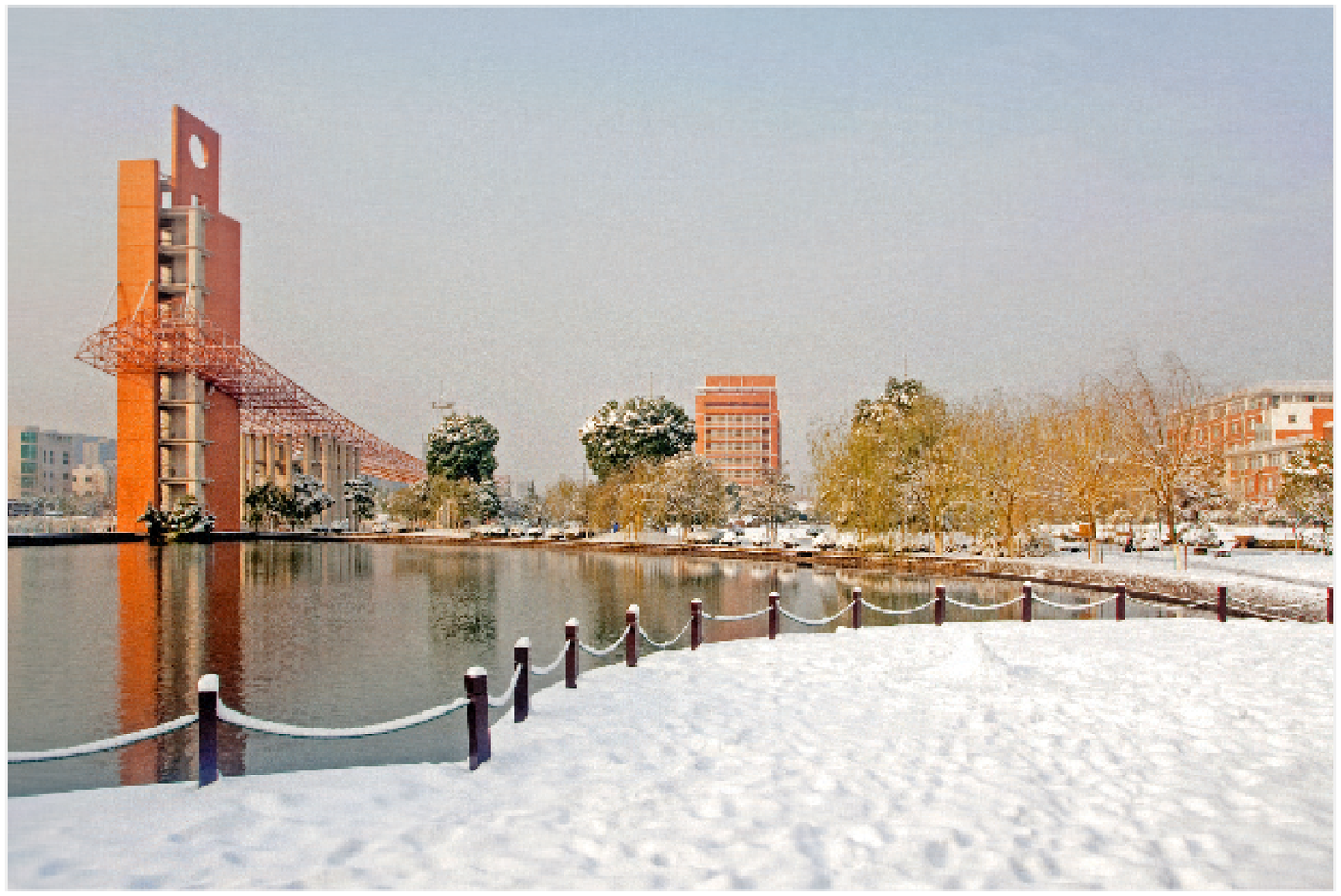} &
			\includegraphics[width=0.3\linewidth]{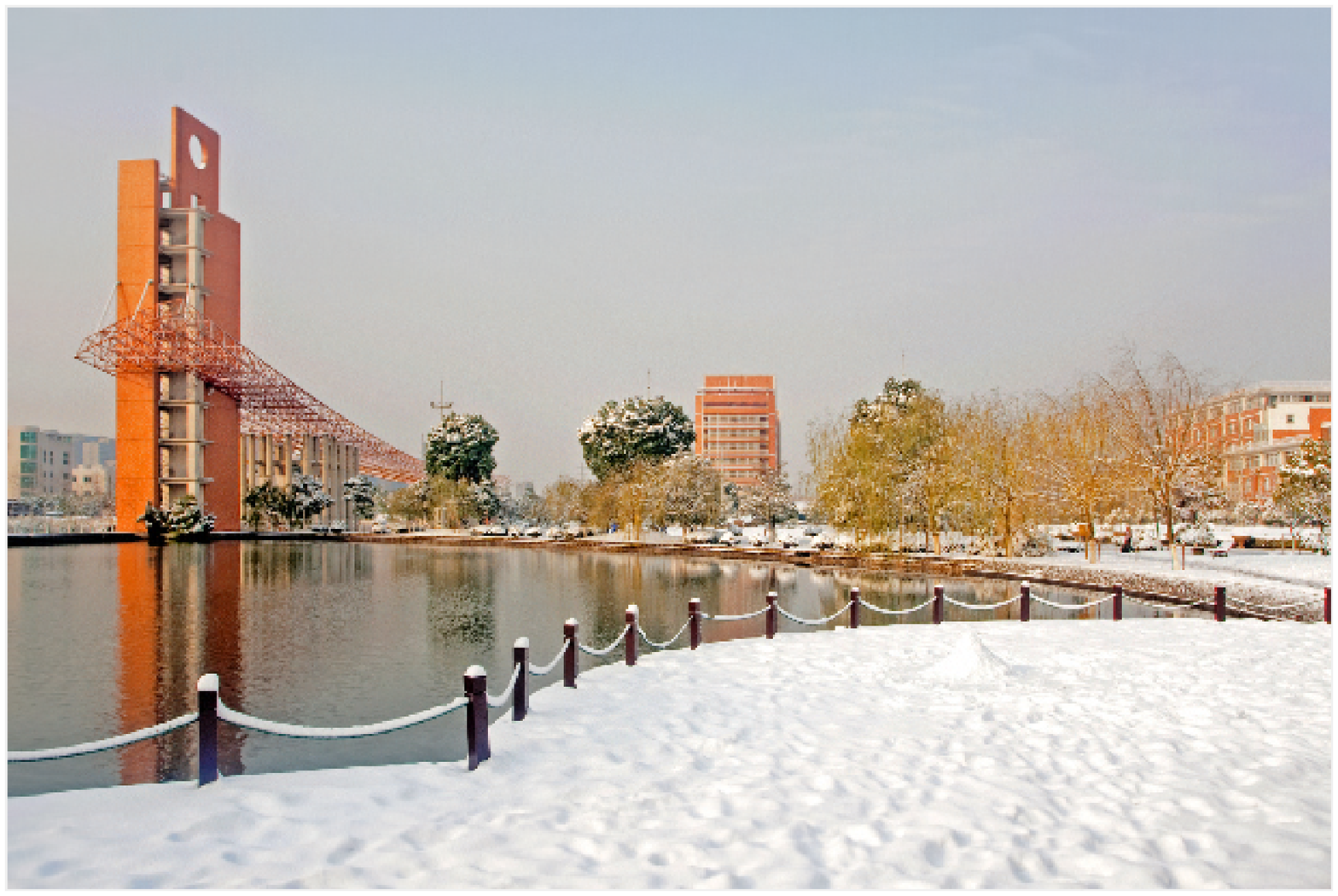} \\	
			{\footnotesize R-STHOSVD } & {\footnotesize Sketch-STHOSVD } & {\footnotesize sub-Sketch-STHOSVD} \vspace{-0.05in} \\
			{\footnotesize (0.61; 34.72)} &
			{\footnotesize (0.55; 34.63)} & {\footnotesize (0.84; 39.97)} \\
		\end{tabular}
		\caption{Results comparison on a HDU picture with a size of $1200\times1800\times3$ in terms of PSNR (i.e., peak signal-to-noise ratio) and CPU time. The target rank is (500,500,3). The two values in e.g. (2.62; 40.61) represent the CPU time and the PSNR, respectively. }
		\label{fig:results-HDU-tensor}
	\end{figure}
	
	\begin{figure}[htb]
		\centering
		\includegraphics[width=0.95\linewidth]{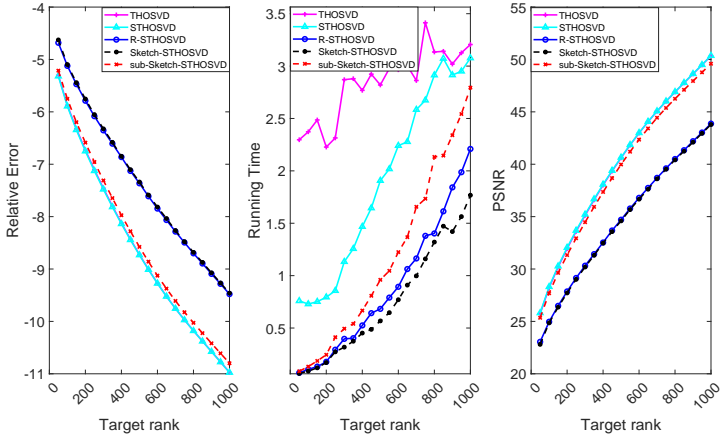}
		\caption{Results comparison on a HDU picture with size of $1200\times1800\times3$ in terms of numerical error (left), CPU time (middle) and PSNR (right). The HDU picture is with target rank $(r,r,3), r\in[50,1000]$. }
		\label{fig:results-HDU}
	\end{figure}
	
	
	\begin{figure}[htb]
		\centering
		\includegraphics[width=0.95\linewidth]{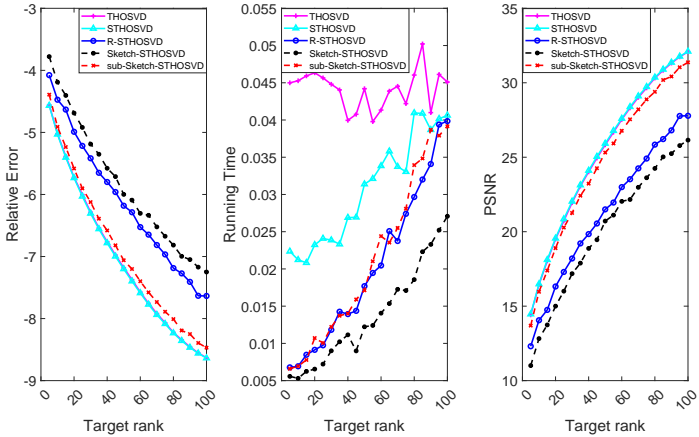}
		\caption{Results comparison on the `hall monitor' grey video with size of $144\times176\times30$ in terms of numerical error (left), CPU time (middle) and PSNR (right). The `hall monitor' grey video is with target rank $(r,r,10)$, $ r\in[5,100]$. }
		\label{fig:results-YUV-tensor}
	\end{figure}
	
	{In the last experiment, a larger-scale real-world tensor data is used. We choose a color image (called the LONDON picture) with a size of $4775\times 7155\times 3$ as the test image and consider the influence of noise. The LONDON picture with white Gaussian noise is generated using the {\tt awgn(X,SNR)} built-in function in \textsc{MATLAB}. We set the target rank as (50,50,3) and SNR to 20. The results comparisons without and with white Gaussian noise are respectively shown in Figure \ref{fig:results-london} and Figure \ref{fig:results-london-noise} in terms of the CPU time and PSNR.}
{Moreover, we also test the algorithms on the LONDON picture as the target rank increases. The  results regarding the relative error, the CPU time and the PSNR are reported in Tables  \ref{london-error}, \ref{london-time} and \ref{london-psnr}, respectively. On the whole, the results again show the consistent performance of the proposed methods.}

	\begin{figure}[htb]
		\centering
		\begin{tabular}{ccc}
			\includegraphics[width=0.3\linewidth]{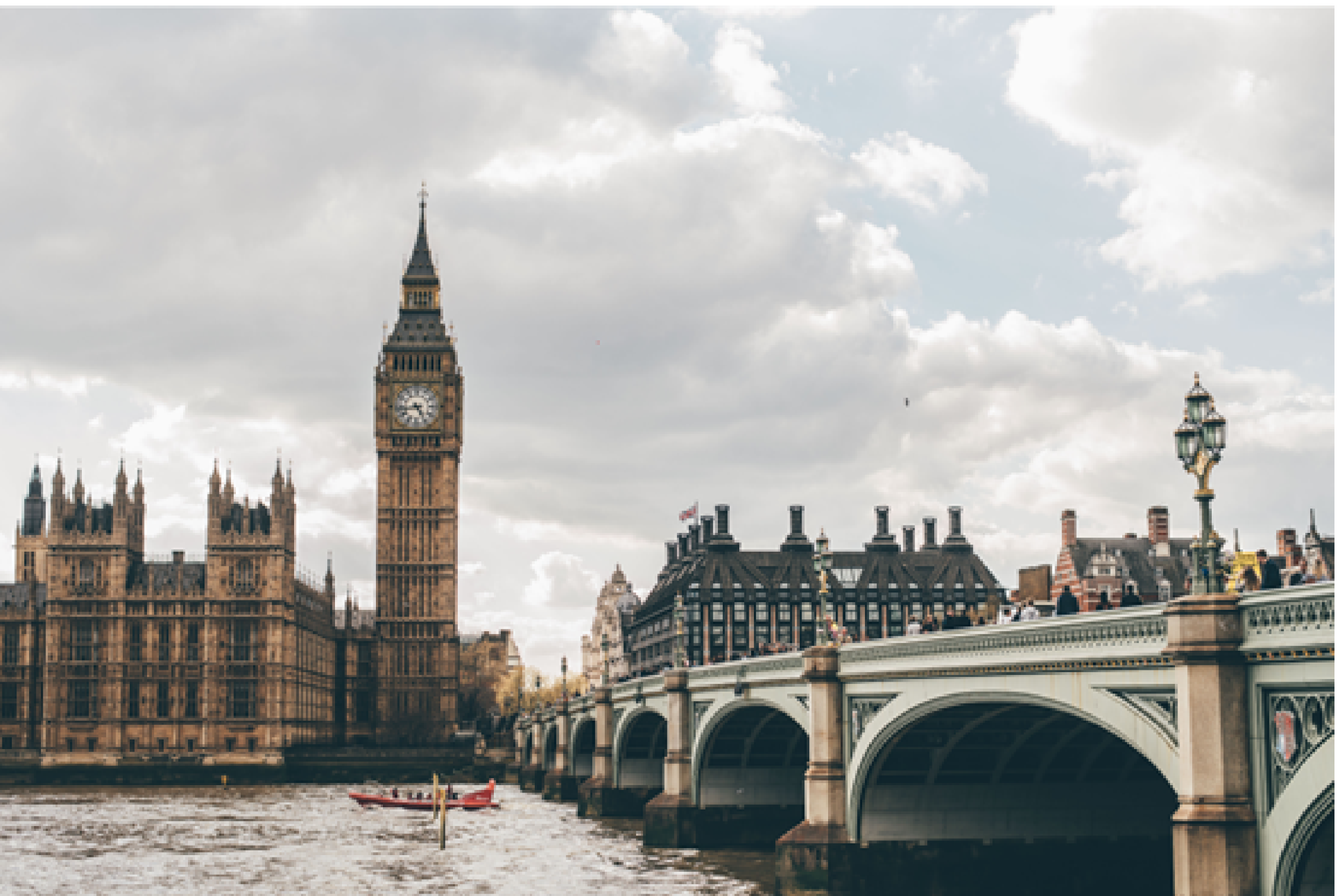} &
			\includegraphics[width=0.3\linewidth]{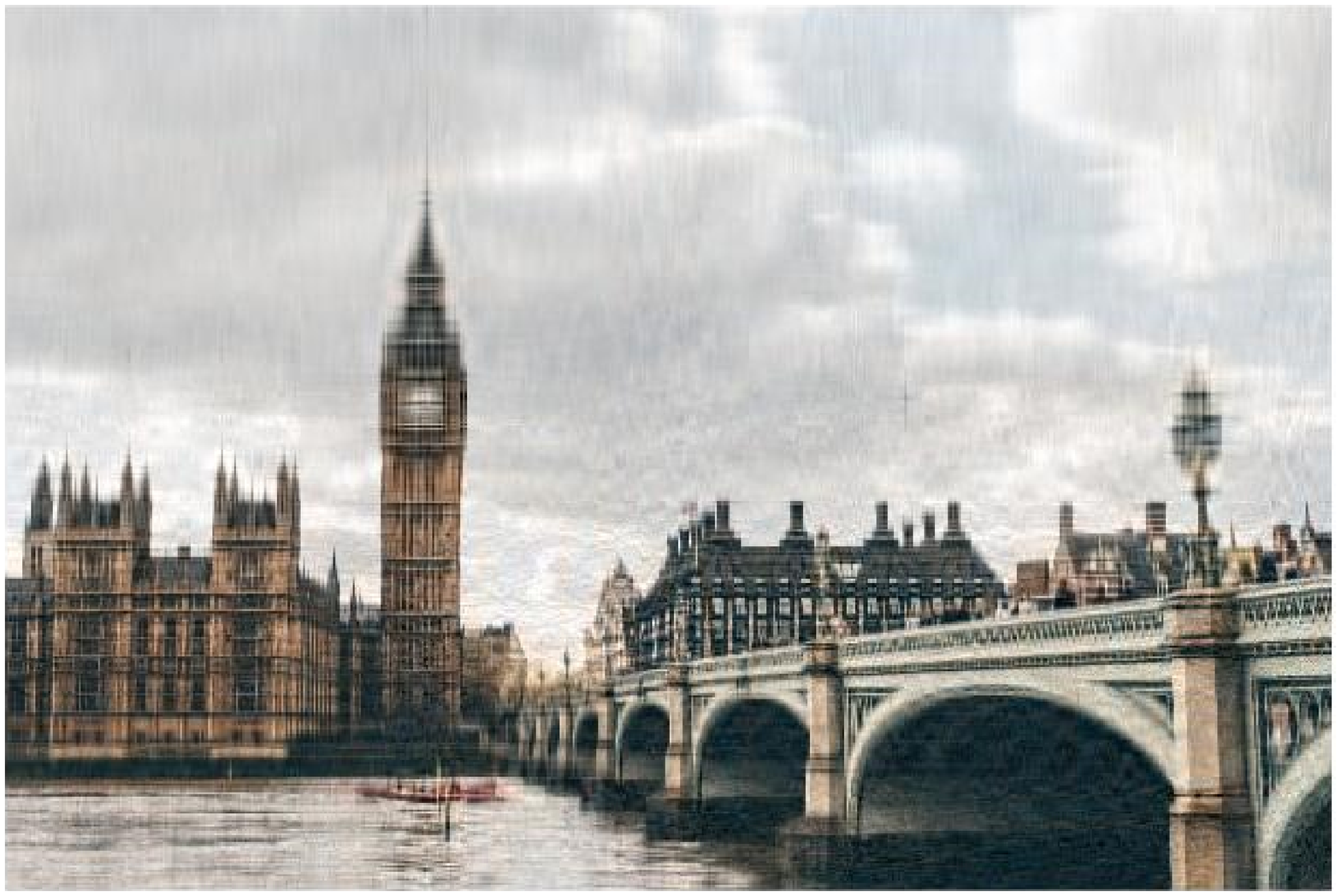} &
			\includegraphics[width=0.3\linewidth]{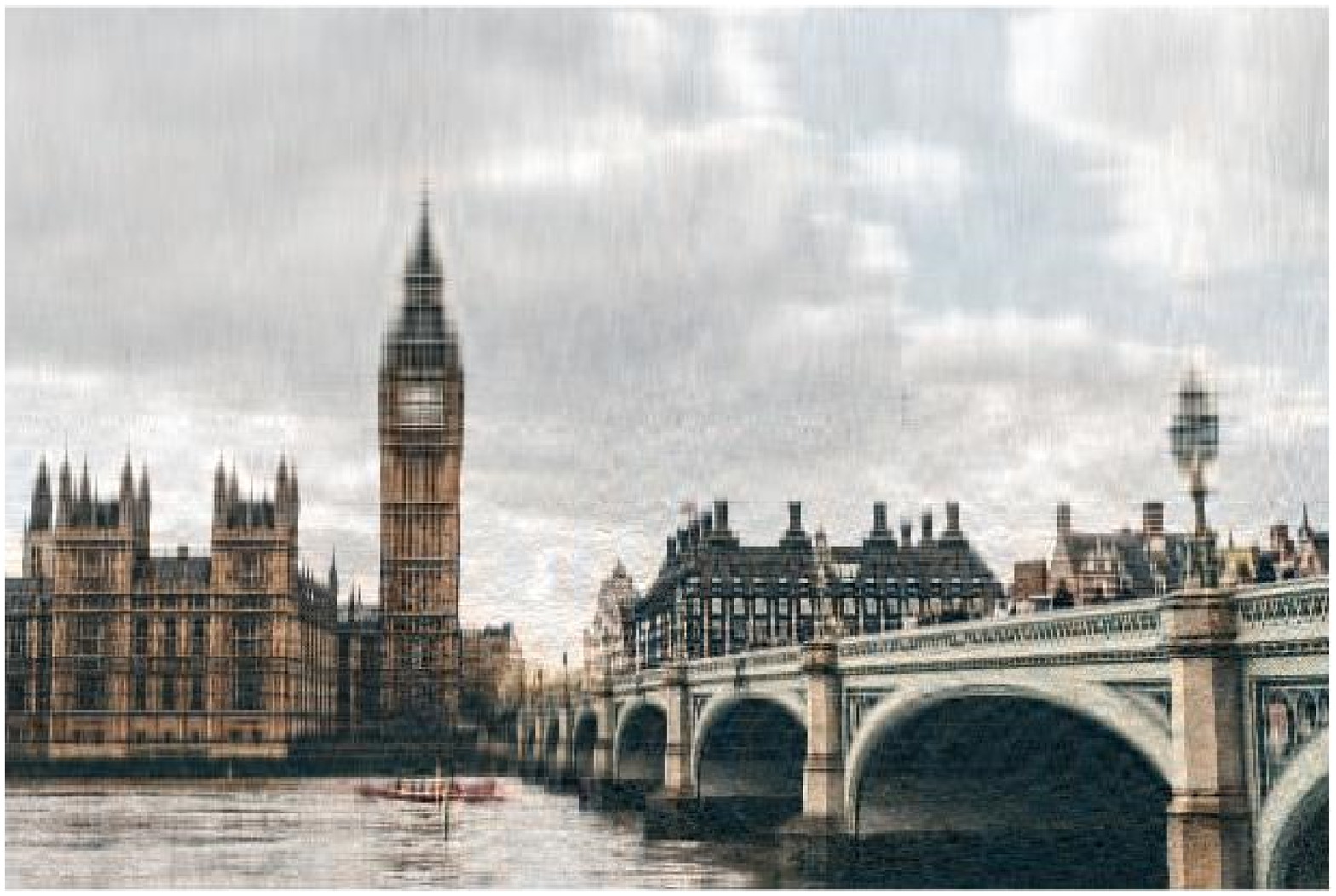} \\
			{\footnotesize Original} & {\footnotesize THOSVD (154.95; 24.07)} & {\footnotesize STHOSVD (49.34; 24.09)} \\
			\includegraphics[width=0.3\linewidth]{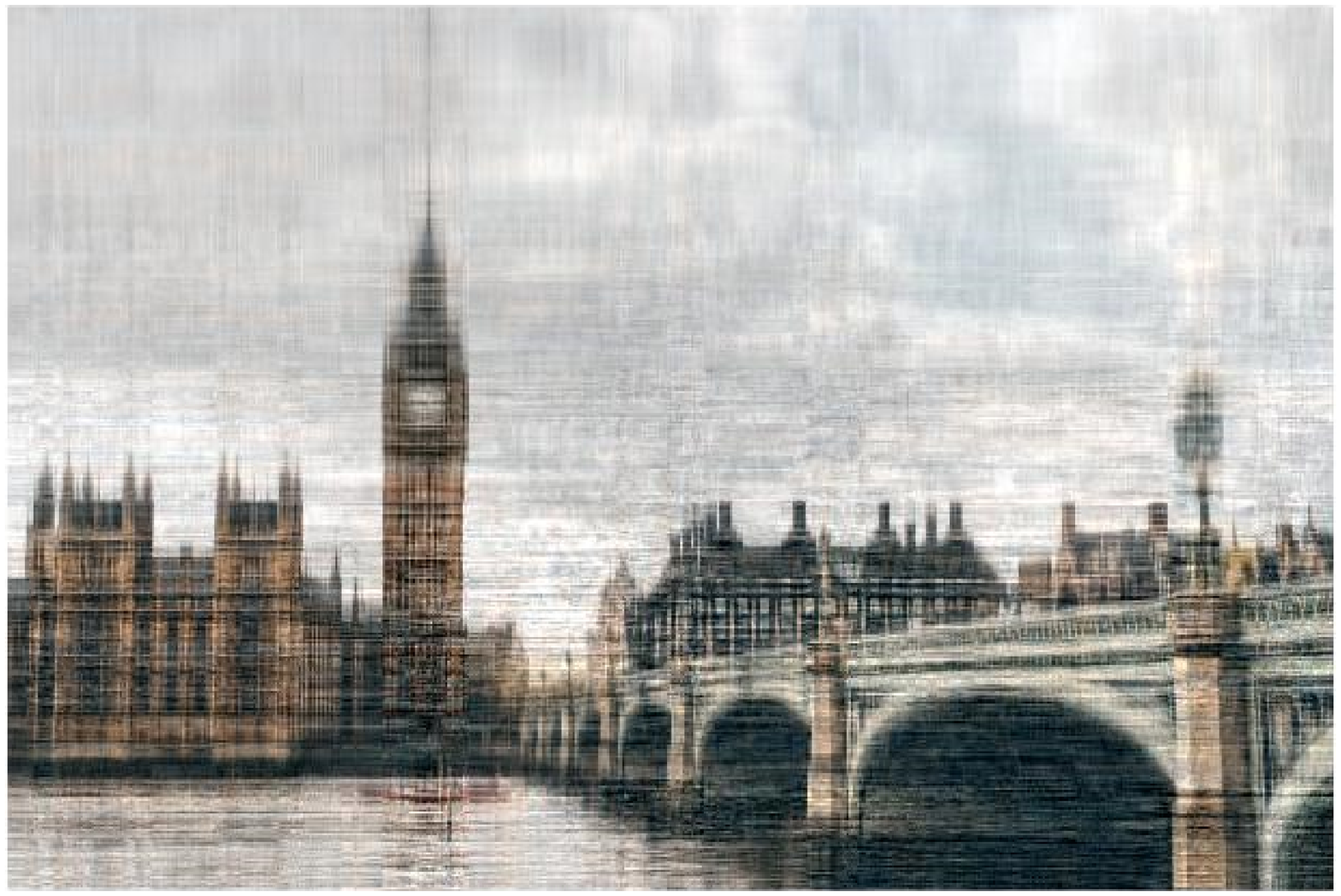} &
			\includegraphics[width=0.3\linewidth]{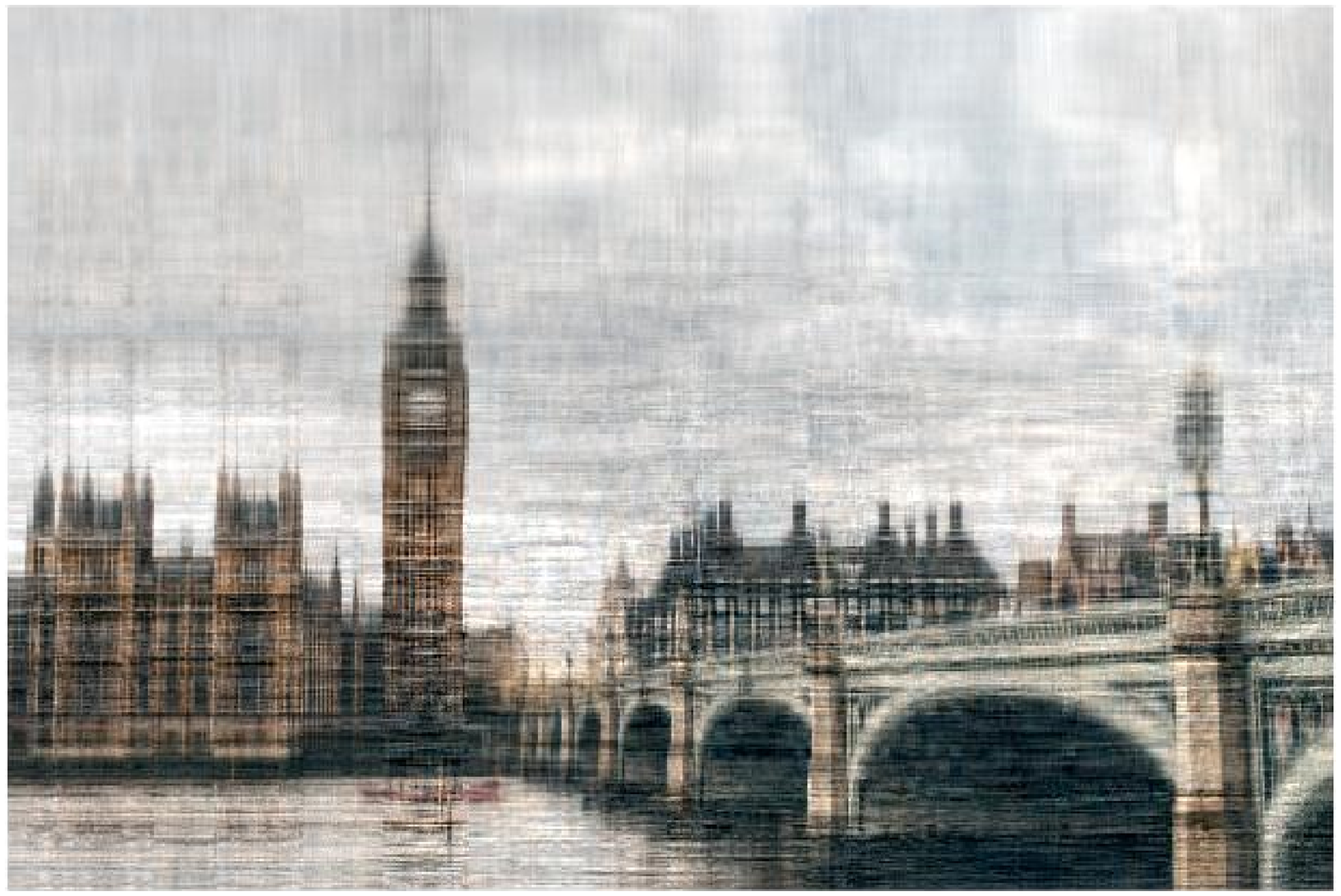} &
			\includegraphics[width=0.3\linewidth]{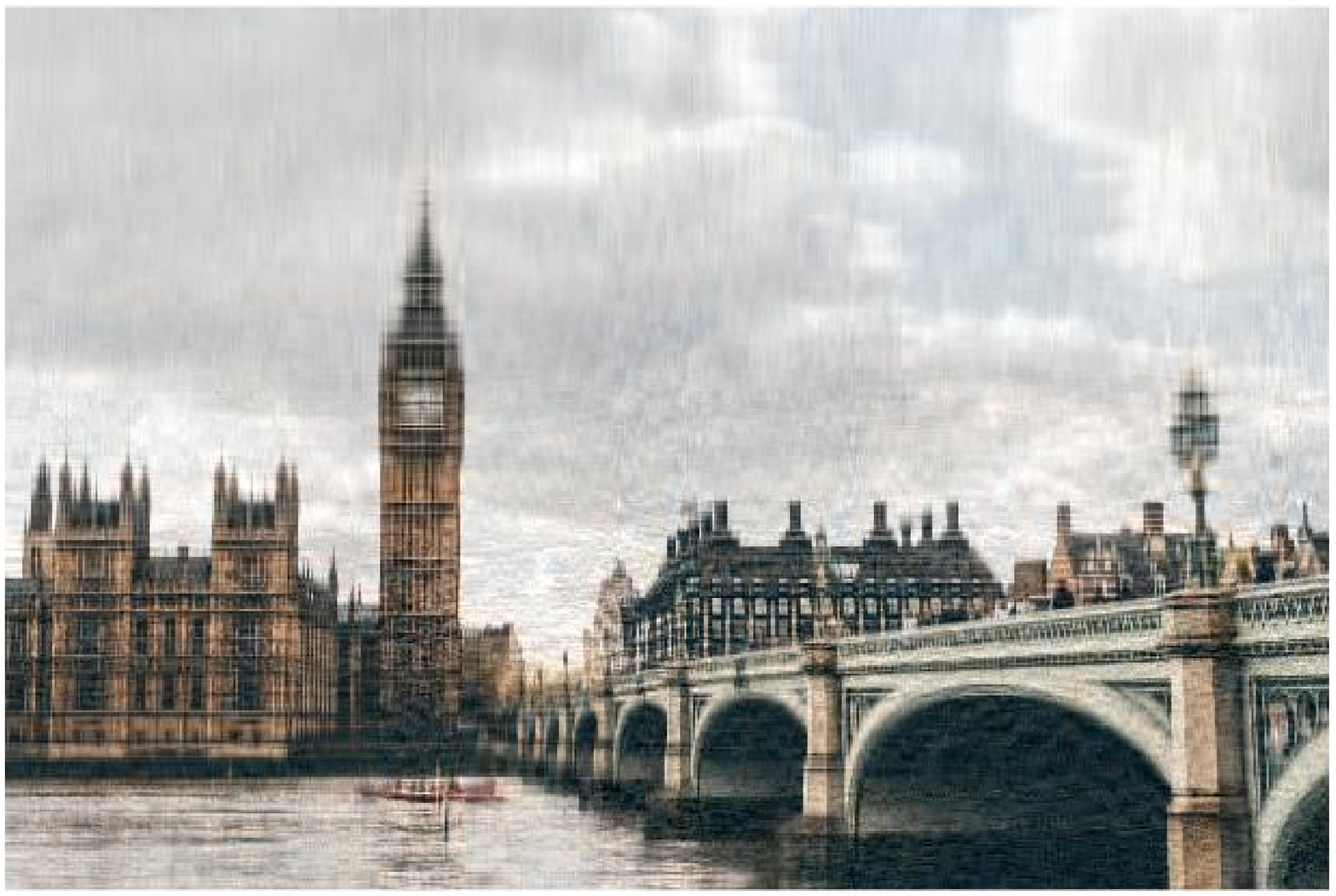} \\	
			{\footnotesize R-STHOSVD } & {\footnotesize Sketch-STHOSVD } & {\footnotesize sub-Sketch-STHOSVD} \vspace{-0.05in} \\
			{\footnotesize (1.29; 21.27)} &
			{\footnotesize (1.17; 21.09)} & {\footnotesize (1.29; 23.65)} \\
		\end{tabular}
		\caption{{Results comparison on LONDON picture with a size of $4775\times 7155\times 3$  in terms of CPU time and PSNR. The target rank is (50,50,3).}} 
	\label{fig:results-london}
\end{figure}

\begin{figure}[htb]
	\centering
	\begin{tabular}{ccc}
		\includegraphics[width=0.3\linewidth]{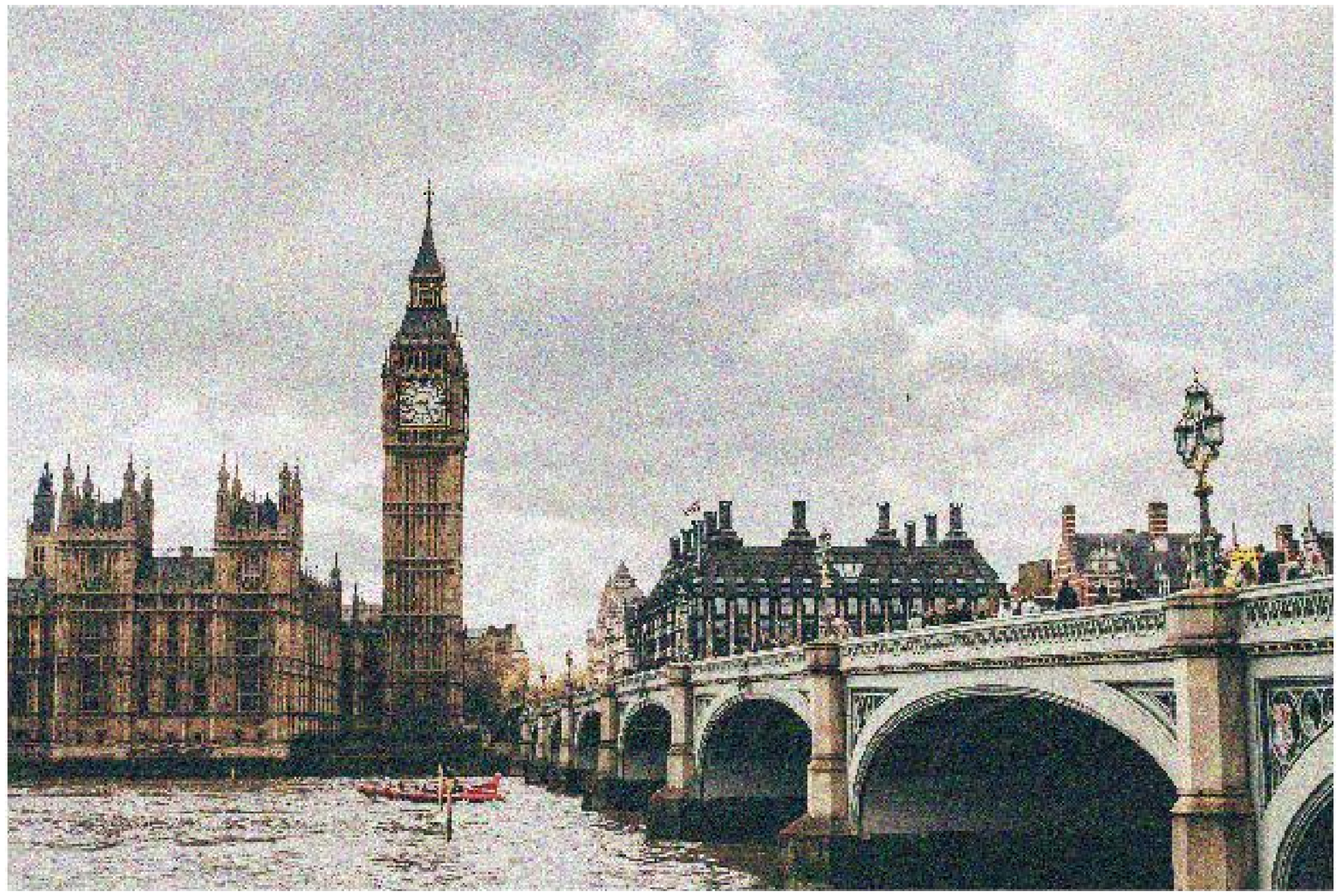} &
		\includegraphics[width=0.3\linewidth]{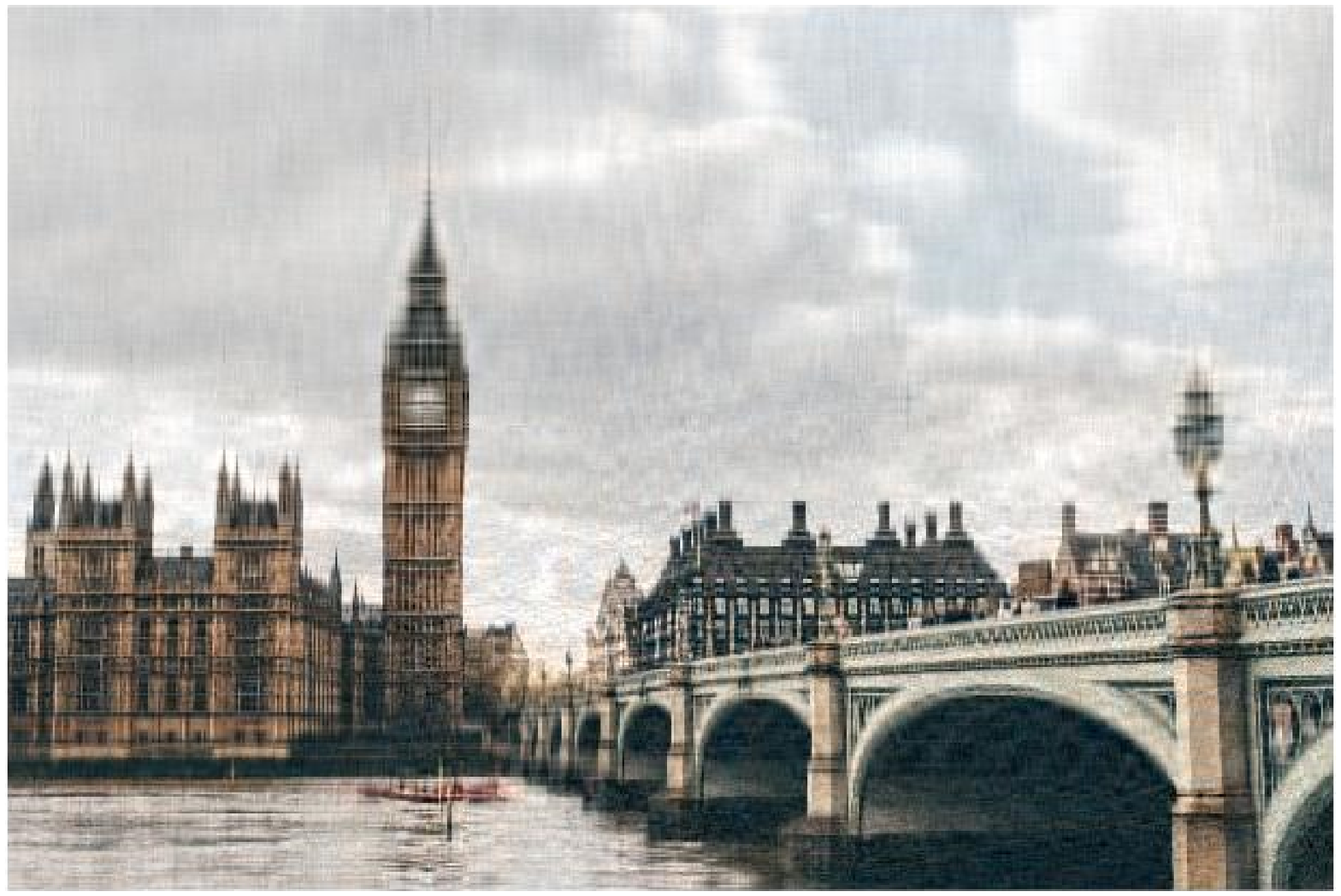} &
		\includegraphics[width=0.3\linewidth]{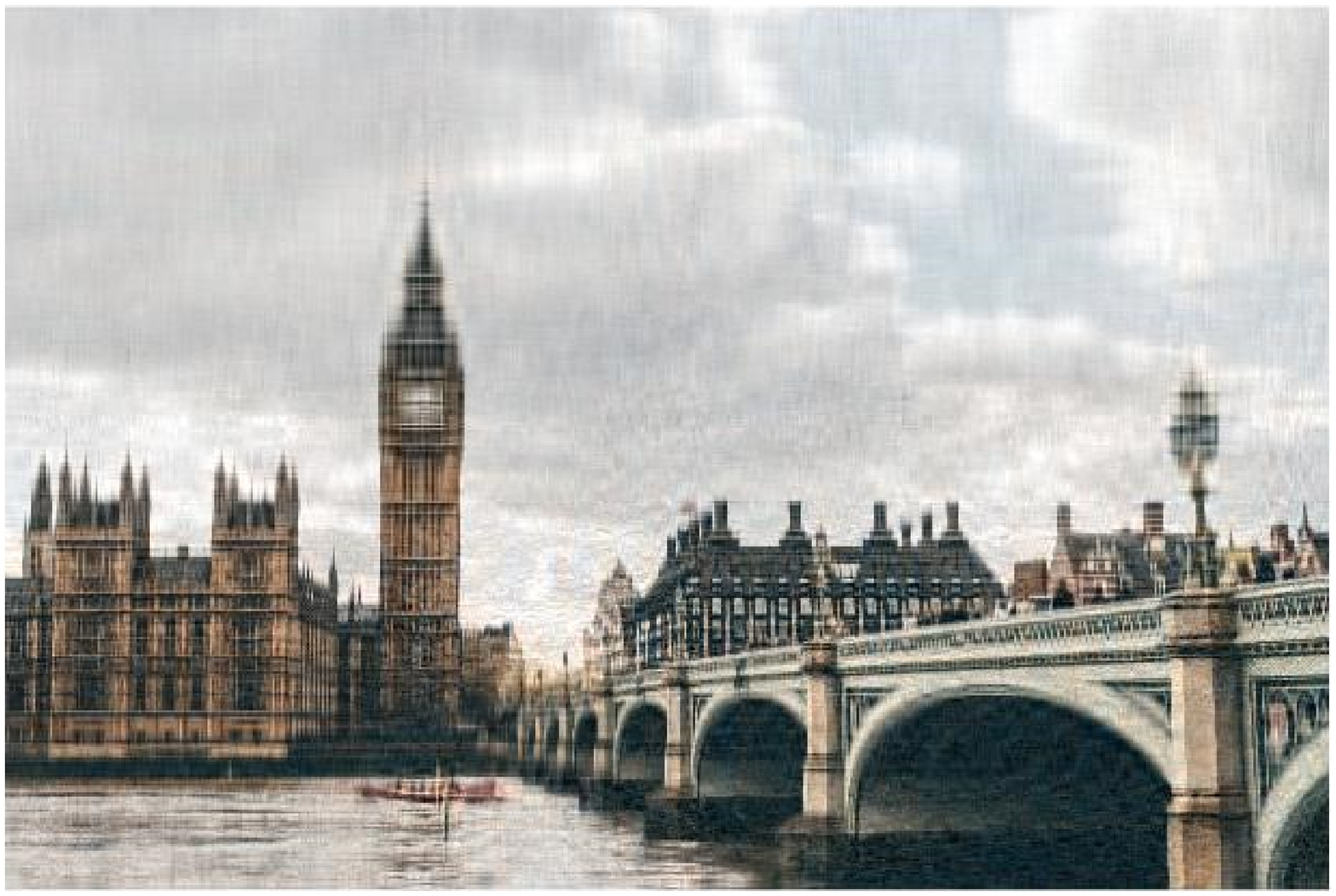} \\
		{\footnotesize Noisy picture(PSNR=16.92)} & {\footnotesize THOSVD (160.59; 20.54)} & {\footnotesize STHOSVD (50.16; 20.54)} \\
		\includegraphics[width=0.3\linewidth]{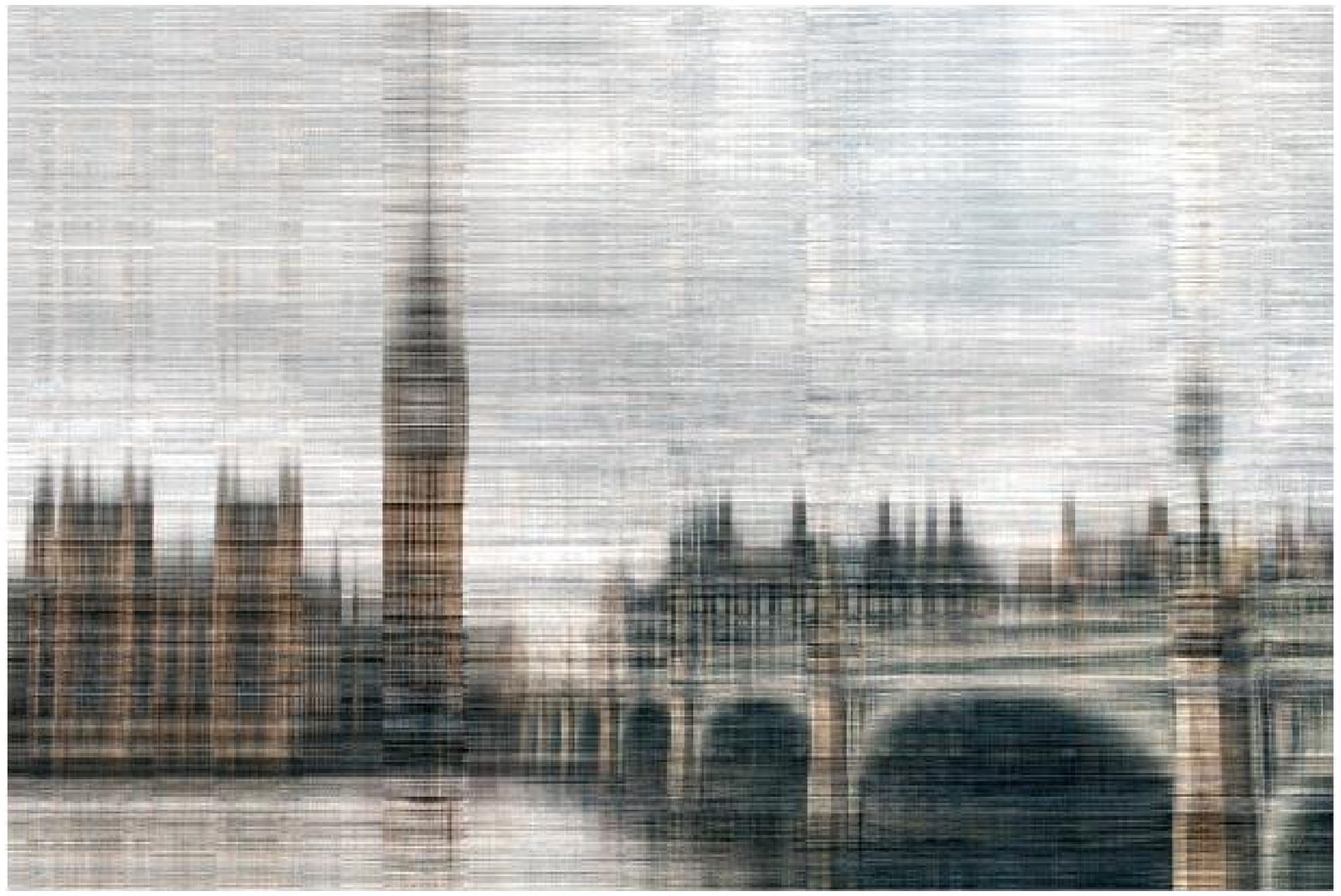} &
		\includegraphics[width=0.3\linewidth]{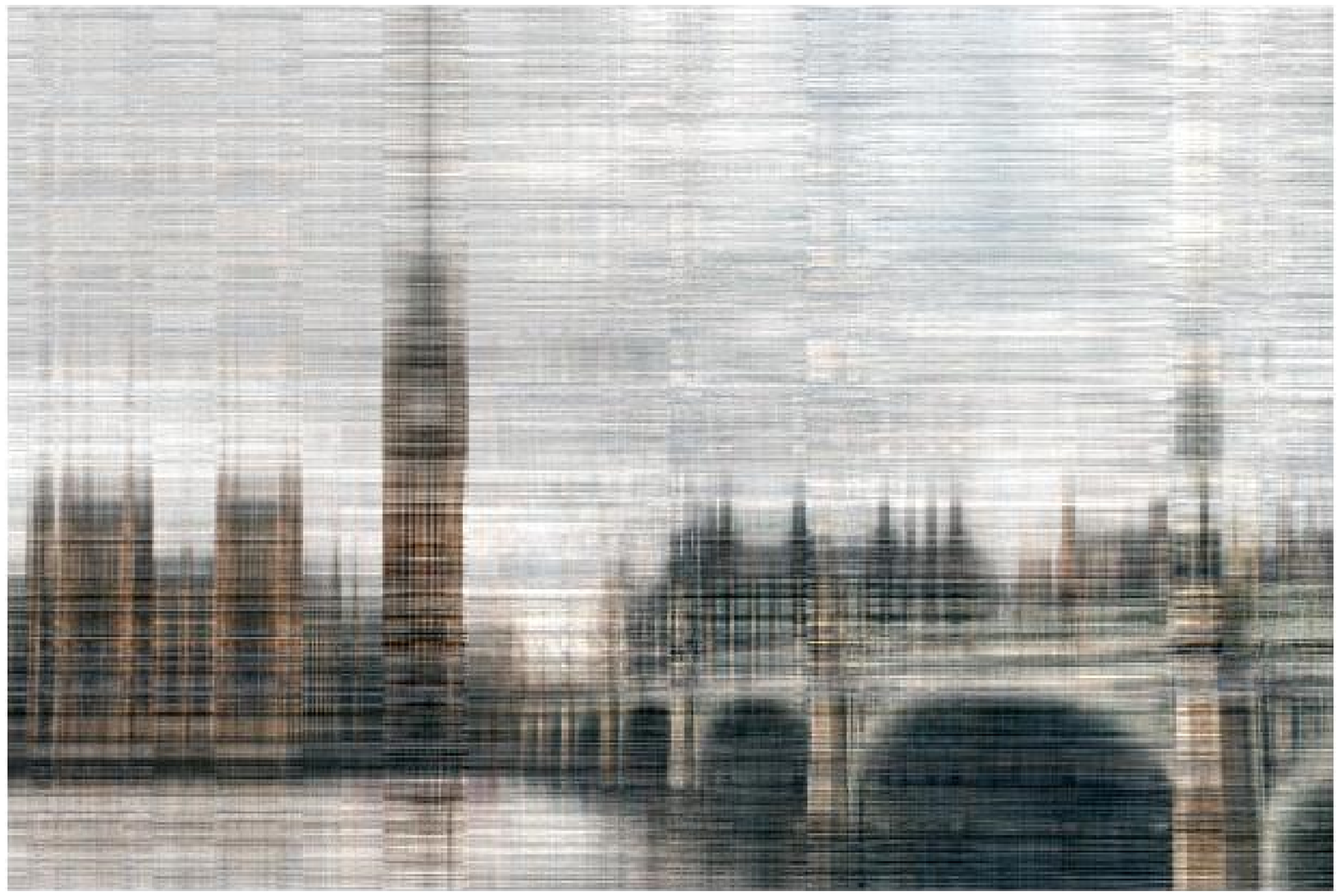} &
		\includegraphics[width=0.3\linewidth]{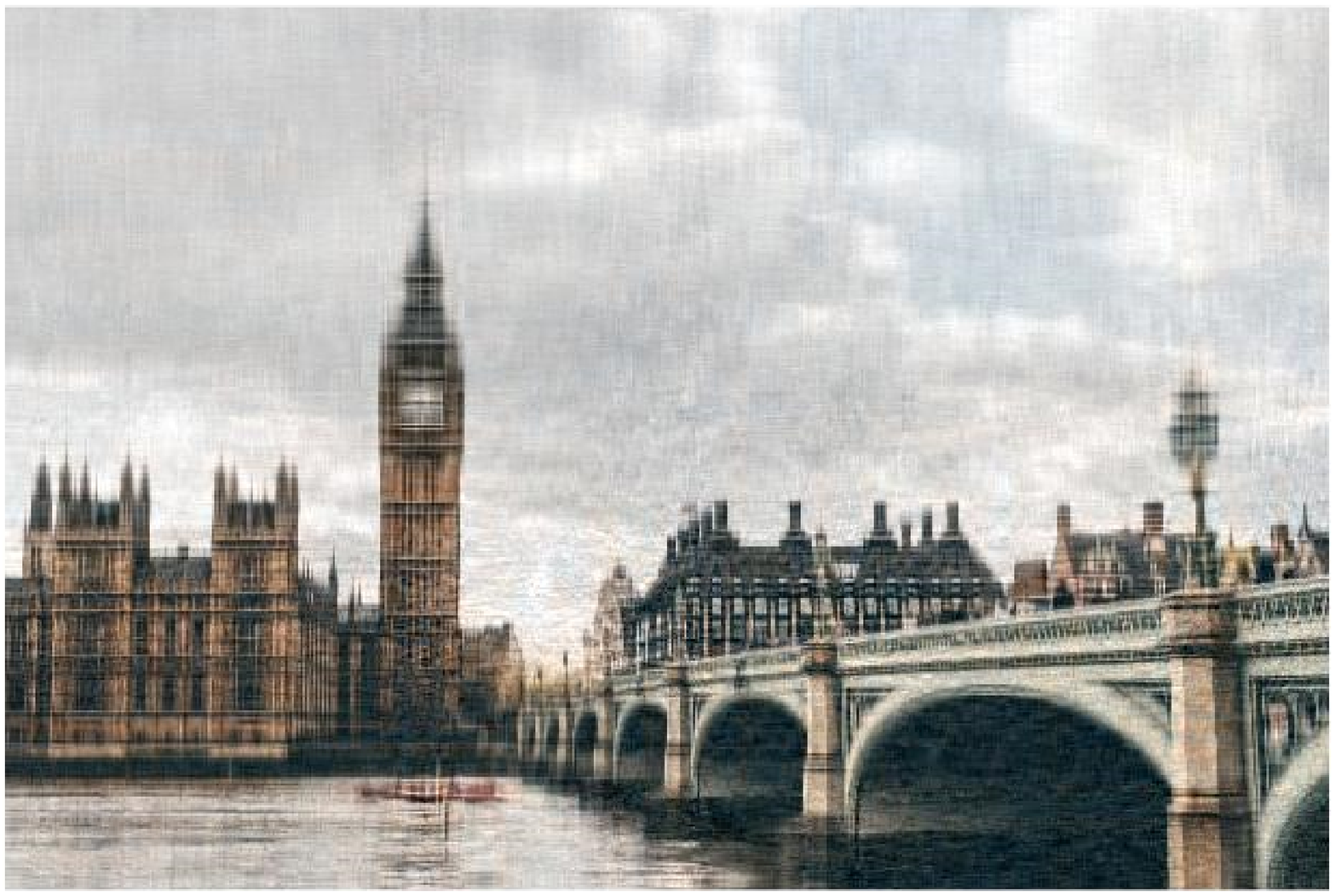} \\	
		{\footnotesize R-STHOSVD } & {\footnotesize Sketch-STHOSVD } & {\footnotesize sub-Sketch-STHOSVD} \vspace{-0.05in} \\
		{\footnotesize (1,25; 19.37)} &
		{\footnotesize (1.15; 19.25)} & {\footnotesize (1.45; 20.45)} \\
	\end{tabular}
	\caption{{Results comparison on LONDON picture  with a size of $4775\times 7155\times 3$ and white Gaussian noise in terms of CPU time and PSNR. The target rank is (50,50,3).}}
\label{fig:results-london-noise}
\end{figure}

\begin{table}[h]
\begin{center}
	\begin{minipage}{\textwidth}
		\caption{{Results comparison in terms of the relative error on the LONDON picture with a size of $4775\times 7155\times 3$ as the target rank increases.}} \label{london-error}
		\tiny{
			\begin{tabular*}
				{\textwidth}{@{\extracolsep{\fill}}lcccccc@{\extracolsep{\fill}}}
				\toprule%
				Target rank & THOSVD  & STHOSVD & R-STHOSVD &Sketch-STHOSVD  &sub-Sketch-STHOSVD\\
				\midrule
				(10,10,10)&0.019037
				&\textbf{0.019025}
				&0.031000
				&0.040006
				&0.020756
				\\
				(20,20,20) & 0.012669
				&\textbf{0.012644}
				&0.023467
				&0.027398
				&0.013703
				\\
				(30,30,30)& 0.010168
				&\textbf{0.010124}
				&0.018354
				&0.020451
				&0.010965
				\\
				(40,40,40) & 0.008630
				&\textbf{0.008599}
				&0.015792
				&0.017029
				&0.009443
				\\
				(50,50,50)& 0.007576
				&\textbf{0.007532}
				&0.013917
				&0.015333
				&0.008286
				\\
				(60,60,60)& 0.006778
				&\textbf{0.006710}
				&0.012967
				&0.013589
				&0.007359
				\\
				(70,70,70)&0.006119
				&\textbf{0.006049}
				&0.011813
				&0.011886
				&0.006687
				\\
				(80,80,80)& 0.005532
				&\textbf{0.005491}
				&0.010658
				&0.011148
				&0.006123
				\\
				(90,90,90)& 0.005076
				&\textbf{0.005023}
				&0.010018
				&0.010378
				&0.005602
				\\
				(100,100,100)& 0.004669
				&\textbf{0.004619}
				&0.009249
				&0.009578
				&0.005172
				\\
				\botrule
		\end{tabular*}}
	\end{minipage}
\end{center}
\end{table}

\begin{table}[h]
\begin{center}
	\begin{minipage}{\textwidth}
		\caption{{Results comparison in terms of the CPU time (in second) on the LONDON picture with a size of $4775\times 7155\times 3$ as the target rank increases.}} \label{london-time}
		\tiny{
			\begin{tabular*}
				{\textwidth}{@{\extracolsep{\fill}}lcccccc@{\extracolsep{\fill}}}
				\toprule%
				Target rank & THOSVD  & STHOSVD & R-STHOSVD &Sketch-STHOSVD  &sub-Sketch-STHOSVD\\
				\midrule
				(10,10,10)& 156.13
				&49.22
				&\textbf{0.94}
				&0.99
				&1.12
				\\
				(20,20,20) & 165.22
				&77.64
				&\textbf{1.24}
				&1.48
				&1.56
				\\
				(30,30,30)& 241.11
				&76.57
				&1.69
				&\textbf{1.39}
				&1.69
				\\
				(40,40,40) & 242.08
				&74.25
				&1.57
				&\textbf{1.45}
				&1.68
				\\
				(50,50,50)& 268.71
				&72.85
				&1.51
				&\textbf{1.45}
				&1.80
				\\
				(60,60,60)& 265.52
				&77.80
				&1.75
				&\textbf{1.51}
				&2.26
				\\
				(70,70,70)& 241.95
				&77.82
				&1.93
				&\textbf{1.78}
				&2.24
				\\
				(80,80,80)& 264.86
				&73.53
				&1.86
				&\textbf{1.74}
				&2.31
				\\
				(90,90,90)& 274.73
				&72.67
				&1.93
				&\textbf{1.83}
				&2.16
				\\
				(100,100,100)& 283.88
				&86.42
				&2.24
				&\textbf{2.20}
				&2.46
				\\
				\botrule
		\end{tabular*}}
	\end{minipage}
\end{center}
\end{table}

\begin{table}[h]
\begin{center}
	\begin{minipage}{\textwidth}
		\caption{{Results comparison in terms of the PSNR on the LONDON picture with a size of $4775\times 7155\times 3$ as the target rank increases.}} \label{london-psnr}
		\tiny{
			\begin{tabular*}
				{\textwidth}{@{\extracolsep{\fill}}lcccccc@{\extracolsep{\fill}}}
				\toprule%
				Target rank & THOSVD  & STHOSVD & R-STHOSVD &Sketch-STHOSVD  &sub-Sketch-STHOSVD\\
				\midrule
				(10,10,10)& 20.06
				&\textbf{20.07}
				&17.96
				&16.86
				&19.70
				\\
				(20,20,20) & 21.84
				&\textbf{21.84}
				&19.18
				&18.51
				&21.50
				\\
				(30,30,30)&22.79
				&\textbf{22.81}
				&20.25
				&19.78
				&22.46
				\\
				(40,40,40) & 23.50
				&\textbf{23.52}
				&20.90
				&20.57
				&23.11
				\\
				(50,50,50)& 24.07
				&\textbf{24.09}
				&21.45
				&21.03
				&23.68
				\\
				(60,60,60)& 24.55
				&\textbf{24.60}
				&21.76
				&21.55
				&24.20
				\\
				(70,70,70)& 25.00
				&\textbf{25.05}
				&22.16
				&22.13
				&24.61
				\\
				(80,80,80)& 25.43
				&\textbf{25.47}
				&22.61
				&22.41
				&25.00
				\\
				(90,90,90)&25.81
				&\textbf{25.85}
				&22.87
				&22.72
				&25.38
				\\
				(100,100,100)& 26.17
				&\textbf{26.22}
				&23.22
				&23.07
				&25.73
				\\
				\botrule
		\end{tabular*}}
	\end{minipage}
\end{center}
\end{table}

{In summary, the numerical results show the superiority of the sub-sketch STHOSVD algorithm for large-scale tensors with or without noise. We can see that sub-Sketch-STHOSVD could achieve close approximations to that of the deterministic algorithms in a time similar to other randomized algorithms.}
\newpage
\section{Conclusion} \label{Sec:con}
In this paper we proposed efficient sketching algorithms, i.e., Sketch-STHOSVD and sub-Sketch-STHOSVD, to calculate the low-rank Tucker approximation of tensors by combining the two-sided sketching technique with the STHOSVD algorithm and using the subspace power iteration. Detailed error analysis is also conducted. Numerical results on both synthetic and real-world data tensors demonstrate the competitive performance of the proposed algorithms in comparison to the state-of-the-art  algorithms.

\section*{Acknowledgements} We would like to thank the anonymous referees for their comments and suggestions on our paper, which lead to great improvements of the presentation. G. Yu's work was supported in part by National Natural Science Foundation of China (No. 12071104) and Natural Science Foundation of Zhejiang Province (No. LD19A010002).

\section*{Appendix}

\begin{lemma}\label{thm5}[\cite{Zhang et al2016}, Theorem 2] Let $\varrho < k-1$ be a positive natural number and $\Omega\in \mathbb{R}^{k\times n}$ be a Gaussian random matrix. Suppose $Q$ is obtained from Algorithm \ref{alg:sub-Sketch}. Then $\forall A\in \mathbb{R}^{m\times n}$, we have
\begin{equation}\label{7}
	\mathbb{E}_{\Omega}\|A-QQ^\top A\|_F^2\leq (1+f(\varrho,k)\varpi_k^{4q})\cdot\tau_{\varrho+1}^2({A}).
\end{equation}
\end{lemma}

\begin{lemma}\label{lem1}[\cite{Tropp et al2017}, Lemma A.3]
Let $A\in \mathbb{R}^{m\times n}$ be an input matrix and $\hat{{A}}=QX$ be the approximation obtained from Algorithm \ref{alg:sub-Sketch}. The approximation error can be decomposed as
\begin{equation}\label{8}
	\|A-\hat{A}\|_F^2=\|A-QQ^\top A\|_F^2+\|X-Q^\top A\|_F^2.
\end{equation}
\end{lemma}

\begin{lemma}\label{lem2}[\cite{Tropp et al2017}, Lemma A.5]
Assume $\Psi\in \mathbb{R}^{l\times n}$ is a standard normal matrix independent from $\Omega$. Then
\begin{equation}\label{9}
	\mathbb{E}_{\Psi}\|X-Q^\top A\|_F^2=f(k,l)\cdot\|A-QQ^\top A\|_F^2.
\end{equation}
\end{lemma}
The error-bound for Algorithm \ref{alg:sub-Sketch} can be shown in Lemma \ref{lemma-thmsub} below.
\begin{lemma}\label{lemma-thmsub}
Assume the sketch size parameter satisfies $l>k+1$. Draw random test matrices $\Omega\in\mathbb{R}^{n\times k}$ and $\Psi{\in\mathbb{R}}^{l\times m}$ independently from the standard normal distribution. Then the rank-$k$ approximation $\hat{{A}}$ obtained from Algorithm \ref{alg:sub-Sketch} satisfies
\begin{equation}\nonumber
	\begin{aligned}
		\mathbb{E}\parallel {A}-\hat{{A}}\parallel_F^2&\le(1+f(k,l))\cdot\min_{\varrho<k-1}(1+f(\varrho,k){\varpi_k}^{4q})\cdot\tau_{\varrho+1}^2({A}).
	\end{aligned}
\end{equation}
\end{lemma}
\begin{proof}
Using equations (\ref{7}), (\ref{8}) and (\ref{9}), we have
\begin{equation*}
	\begin{aligned}
		\mathbb{E}\parallel A-\hat{{A}}\parallel_F^2&=\mathbb{E}_{\Omega}\|A-QQ^\top A\|_F^2+\mathbb{E}_{\Omega}\mathbb{E}_{\Psi}\|X-Q^\top A\|_F^2\\
		& =(1+f(k,l))\cdot\mathbb{E}_{\Omega}\|A-QQ^\top A\|_F^2\\
		& \leq(1+f(k,l))\cdot(1+f(\varrho,k){\varpi_k}^{4q})\cdot\tau_{\varrho+1}^2({A}). 		
	\end{aligned}
\end{equation*}
After minimizing over eligible index $\varrho<k-1$, the proof is completed.
\end{proof}

We are now in the position to prove Theorem \ref{thmsubs}. Combining Theorem \ref{thmST} and Lemma \ref{lemma-thmsub}, we have
\begin{equation*}
\begin{aligned}
	& \ \mathbb{E}_{\{\Omega _j\}_{j = 1}^N}\|\mathcal {X} - \widehat {\mathcal {X}}\|_F^2  \\
	= & \sum \limits_{n = 1}^N\mathbb{E}_{\{\Omega _j\}_{j = 1}^N} \|\hat{\mathcal {X}}^{(n - 1)} - \hat{\mathcal {X}}^{(n )}\|_F^2\\
	= & \sum \limits_{n = 1}^N \mathbb{E}_{\{\Omega _j\}_{j = 1}^{n-1}}\left\{\mathbb{E}_{\Omega _n}\|\hat{\mathcal {X}}^{(n - 1)} - \hat{\mathcal {X}}^{(n)}\|_F^2 \right\}\\
	= & \sum \limits_{n = 1}^N \mathbb{E}_{\{\Omega _j\}_{j = 1}^{n-1}}\left\{\mathbb{E}_{\Omega _n}\|\mathcal {G}^{(n - 1)} \times _{i = 1}^{n - 1} {{U}^{(i)}}{ \times _n}({I} - {{U}^{(n)}}{U}^{(n)\top})
	\|_F^2 \right\}\\
	\le & \sum \limits_{n = 1}^N \mathbb{E}_{\{\Omega _j\}_{j = 1}^{n-1}}\left\{\mathbb{E}_{\Omega _n}\|({I} - {{U}^{(n)}}{U}^{(n)\top}){G}_{(n)}^{(n-1)})
	\|_F^2 \right\}\\
	\le & \sum \limits_{n = 1}^N \mathbb{E}_{\{\Omega _j\}_{j = 1}^{n-1}}(1+f(r_n,l_n))\cdot\min_{\varrho_n<r_n-1}(1+f(\varrho_n,r_n){\varpi_r}^{4q})\sum \limits_{i = r_n+1}^{I_n}\sigma_{i}^{2}({G}_{(n)}^{(n-1)})\\
	\le & \sum\limits_{n = 1}^N \mathbb{E}_{\{\Omega _j\}_{j = 1}^{n-1}} (1+f(r_n,l_n))\cdot\min_{\varrho_n<r_n-1}(1+f(\varrho_n,r_n){\varpi_r}^{4q}) \Delta _n^2({\cal X})\\
	= & \sum\limits_{n = 1}^N (1+f(r_n,l_n))\cdot\min_{\varrho_n<r_n-1}(1+f(\varrho_n,r_n){\varpi_r}^{4q})\Delta _n^2({\cal X})\\
	\le & \sum\limits_{n = 1}^N (1+f(r_n,l_n))\cdot\min_{\varrho_n<r_n-1}(1+f(\varrho_n,r_n){\varpi_r}^{4q}) \|\mathcal{X}-\hat{\mathcal{X}}_{\rm opt}\|_F^2 \ ,
\end{aligned}
\end{equation*}
which completes the proof of Theorem \ref{thmsubs}.



\end{document}